\documentclass{amsart}%
\usepackage{amssymb}
\usepackage{amsfonts}
\usepackage{geometry}
\usepackage{graphicx}
\usepackage{amsmath}%
\providecommand{\U}[1]{\protect\rule{.1in}{.1in}}
\newtheorem{theorem}{Theorem}
\theoremstyle{plain}

\newtheorem{conclusion}[theorem]{Conclusion}

\newtheorem{definition}[theorem]{Definition}

\newtheorem{lemma}[theorem]{Lemma}
\newtheorem{notation}[theorem]{Notation}

\newtheorem{remark}[theorem]{Remark}

\numberwithin{equation}{section}
\begin{document}
\title[Fefferman's Fourier extension theorem in the plane]{A discussion of two arguments related to Fefferman's Fourier extension theorem
in the plane}
\author[E. T. Sawyer]{Eric T. Sawyer$^{\dagger}$}
\address{Eric T. Sawyer, Department of Mathematics and Statistics\\
McMaster University\\
1280 Main Street West\\
Hamilton, Ontario L8S 4K1 Canada}
\email{sawyer@mcmaster.ca}
\thanks{$\dagger$ Research supported in part by a grant from the National Science and
Engineering Research Council of Canada.}
\date{\today }

\begin{abstract}
The Fourier extension conjecture of E. Stein, see \cite{Ste} and \cite{Ste2},
in $2$ dimensions was,%
\[
\left\Vert \widehat{f\sigma}\right\Vert _{L^{p}\left(  \mathbb{R}^{2}\right)
}\leq C_{p}\left\Vert f\right\Vert _{L^{p}\left(  \sigma\right)  },\ \text{for
}f\in L^{p}\left(  \sigma\right)  \text{ and }p>4,
\]
where $\sigma$ is arc length measure on the circle $\mathbb{T}$ or a smooth
compactly supported measure on the parabola $\mathbb{P}$.

This conjecture was proved in 1973 by C.Fefferman \cite{Fef}, see also
Carleson and Sj\"{o}lin \cite{CaSj} and Zygmund\ \cite{Zyg}, with
simplifications given by other authors later on, in particular by L.
H\"{o}rmander and T. Tao, see e.g. the book by C. Demeter \cite{Dem}.

We discuss yet two\ more arguments for this classical theorem on the parabola.
The first argument uses C. Fefferman's decoupling \cite{Fef} together with a
decomposition into Haar wavelets. This sets the stage for the second argument
whose point of departure is the bilinear characterization of Tao, Vargas and
Vega \cite{TaVaVe}, and relies on the bilinear interplay of the classical wave
packet constructions and discrete characterizations with an induction on scales.

Each of the above two arguments rely on some form of Fefferman's convolution
decoupling and the special nature of the critical planar index $4$ as a
positive even integer, but the second might~be more amenable to weighted generalizations.

\end{abstract}
\maketitle
\tableofcontents

\section{Introduction}

In this paper we discuss two arguments in connection with Fefferman's Fourier
restriction theorem on the parabola, equivalently its dual formulation as a
Fourier extension theorem.

The first argument, which is hardly new, uses C. Fefferman's `half decoupling'
arising from pairwise disjointedness of certain convolutions, followed by the
remaining `half decoupling' that uses Haar wavelets and simple estimates of
the resulting convolutions. Both of these decouplings rely heavily on the
curvature of the paraobola and the special nature of the critical planar index
$4$ as a positive \emph{even} integer.

Then we sketch a second argument whose point of departure is the bilinear
characterization of the Fourier extension conjecture in Tao, Vargas and Vega
\cite{TaVaVe}. This second argument exploits the use of bilinear testing
conditions in conjunction with wave packet decompositions to give a proof of
the bilinear estimate in \cite{BeCaTa} for the parabola. The curvature of the
parabola now arises through a variant of Fefferman's convolution `half
decoupling', and the separation of normals in the bilinear estimate. Again the
special nature of $4$ is is crucially used, but this second argument may be
more amenable to weighted generalizations.

\subsection{A brief history}

The Fourier extension conjecture in $d\geq2$ dimensions arose from unpublished
work of E. Stein in 1967, see e.g. \cite[see the Notes at the end of Chapter
IX, p. 432, where Stein proved the restriction conjecture for $1\leq
p<\frac{4d}{3d+1}$]{Ste2} and \cite{Ste},%
\[
\left\Vert \widehat{f\sigma_{d-1}}\right\Vert _{L^{p}\left(  \mathbb{R}%
^{d}\right)  }\leq C_{p}\left\Vert f\right\Vert _{L^{p}\left(  \sigma
_{d-1}\right)  },\ \text{for }f\in L^{p}\left(  \sigma_{d-1}\right)  \text{
and }p>\frac{2d}{d-1},
\]
where $\sigma_{d-1}$ is surface measure on the sphere $\mathbb{S}^{d-1}$ or a
smooth compactly supported measure on the paraboloid $\mathbb{P}^{d-1}$. The
case $d=2$ of the Fourier extension conjecture for the circle was proved over
half a century ago by C. Fefferman \cite{Fef}, see also A. Zygmund \cite{Zyg}
and L. Carleson and P. Sj\"{o}lin \cite{CaSj}.

\section{Fefferman's decoupling argument}

We start with a bounded function $f$ supported in $\left[  0,1\right]  $ and
push the absolutely continuous measure $f\left(  x\right)  dx$ on the real
line forward under the map $\Phi\left(  x\right)  =\left(  x,x^{2}\right)  $
to a singular measure $\Phi_{\ast}f$ in the plane supported on the unit
parabola $\mathbb{P}$. The Fourier extension of $f$ is then defined by%
\[
\mathcal{E}f\left(  \xi\right)  \equiv\widehat{\Phi_{\ast}f}\left(
\xi\right)  =\int_{0}^{1}e^{-i\left(  \xi_{1}x+\xi_{2}x^{2}\right)  }f\left(
x\right)  dx,
\]
and Stein's conjecture in the plane, apart from the Knapp line, is equivalent
to%
\[
\left\Vert \mathcal{E}f\right\Vert _{L^{q}\left(  \mathbb{R}^{2}\right)  }\leq
C_{p}\left\Vert f\right\Vert _{L^{q}\left(  \left[  0,1\right]  \right)
},\ \ \ \ \ \text{for }q>4\text{ and }\operatorname*{Supp}f\subset\left[
0,1\right]  .
\]

By subtracting a constant, we may suppose that $f$ has mean zero on
$I_{0}\equiv\left[  0,1\right]  $, and hence that the Haar expansion of $f$ in
the standard grid $\mathcal{D}$ is given by%
\begin{align}
f  &  =\sum_{I\subset I_{0}}\left\langle f,h_{I}\right\rangle h_{I}=\sum
_{t=0}^{\infty}\mathsf{P}_{t}f,\ \ \ \ \ \text{where }\mathsf{P}_{t}%
f\equiv\sum_{I\subset\mathcal{D}_{t}\left[  I_{0}\right]  }\left\langle
f,h_{I}\right\rangle h_{I}\ ,\label{Haar decomp}\\
f  &  =\lim_{s\rightarrow\infty}\mathsf{Q}_{s}f,\ \ \ \ \ \text{where
}\mathsf{Q}_{s}f\equiv\sum_{s=0}^{s}\mathsf{P}_{t}f=\sum_{I\subset
\mathcal{D}\left[  I_{0}\right]  :\ \ell\left(  I\right)  \geq2^{-s}%
}\left\langle f,h_{I}\right\rangle h_{I}\ ,\nonumber
\end{align}
where $\mathcal{D}_{t}\left[  I_{0}\right]  \equiv\left\{  I\in\mathcal{D}%
:\ell\left(  I\right)  =2^{-t}\right\}  $.

In particular the projection $\mathsf{Q}_{s}f$ is constant on each interval
$I_{n}\equiv\left[  2^{-s}\left(  n-1\right)  ,2^{-s}n\right)  $ in
$\mathcal{D}_{s}\left[  I_{0}\right]  $. For convenience we fix $s\in
\mathbb{N}$ and write
\begin{equation}
F=\sum_{n=1}^{2^{s}}F_{n},\ \ \ \ \ \text{ where }F_{n}\equiv a_{n}%
\mathbf{1}_{I_{n}}.\label{constant}%
\end{equation}
Now let $\varphi$ be a nonnegative smooth function adapted to the unit disk in
$\mathbb{R}^{2}$ with integral $1$. Then if $\varphi_{s}\left(  z\right)
\equiv\frac{2^{4s}}{c^{2}}\varphi\left(  \frac{z}{c2^{-2s}}\right)  $ and
$q>4$, we have from the Hausdorff-Young inequality with exponents $\frac{q}{2}
$ and $\left(  \frac{q}{2}\right)  ^{\prime}=\frac{q}{q-2}$ that
\begin{align*}
&  \int_{B\left(  0,2^{2s}\right)  }\left\vert \mathcal{E}F\left(  \xi\right)
\right\vert ^{q}d\xi\lesssim\int_{\mathbb{R}^{2}}\left\vert \widehat
{\varphi_{s}}\left(  \xi\right)  \right\vert ^{q}\left\vert \sum_{n=1}^{2^{s}%
}\widehat{\Phi_{\ast}F_{n}}\left(  \xi\right)  \right\vert ^{q}d\xi
=\int_{\mathbb{R}^{2}}\left\vert \sum_{m,n=1}^{2^{s}}\widehat{\varphi_{s}%
\ast\Phi_{\ast}F_{m}}\left(  \xi\right)  \widehat{\varphi_{s}\ast\Phi_{\ast
}F_{n}}\left(  \xi\right)  \right\vert ^{\frac{q}{2}}d\xi\\
&  =\int_{\mathbb{R}^{2}}\left\vert \sum_{m,n=1}^{2^{s}}\left(  \varphi
_{s}\ast\Phi_{\ast}F_{m}\ast\varphi_{s}\ast\Phi_{\ast}F_{n}\right)  ^{\wedge
}\left(  \xi\right)  \right\vert ^{\frac{q}{2}}d\xi\leq\left(  \int
_{\mathbb{R}^{2}}\left\vert \sum_{m,n=1}^{2^{s}}\varphi_{s}\ast\Phi_{\ast
}F_{m}\ast\varphi_{s}\ast\Phi_{\ast}F_{n}\left(  z\right)  \right\vert
^{\left(  \frac{q}{2}\right)  ^{\prime}}dz\right)  ^{\frac{\frac{q}{2}%
}{\left(  \frac{q}{2}\right)  ^{\prime}}}\\
&  =\left(  \int_{\mathbb{R}^{2}}\left\vert \sum_{m,n=1}^{2^{s}}G_{m}\ast
G_{n}\left(  z\right)  \right\vert ^{\frac{q}{q-2}}dz\right)  ^{\frac{q-2}{2}%
},\ \ \ \ \ \ \ \ \ \ \ \ \ \ \ \text{where }G_{k}\equiv\varphi_{s}\ast
\Phi_{\ast}F_{k}\ .
\end{align*}

Of course $\sum_{m=1}^{2^{s}}\mathbf{1}_{\operatorname*{Supp}G_{m}}\leq C$ by
definition of $G_{m}$, and Fefferman's clever decoupling observation in
\cite{Fef} is that%
\begin{equation}
\sum_{m,n=1}^{2^{s}}\mathbf{1}_{\operatorname*{Supp}G_{m}\ast G_{n}}\leq
\sum_{m,n=1}^{2^{s}}\mathbf{1}_{\operatorname*{Supp}G_{m}+\operatorname*{Supp}%
G_{n}}\leq C,\label{Feff over}%
\end{equation}
independent of $s$, by the curvature of the parabola. This can be seen by
letting $R_{m}$ be a tilted $C2^{-s}\times C2^{-2s}$ rectangle containing the
support of $G_{m}$, and then carefully examining the locations of the set sums
$R_{m}+R_{n}$ in the plane.

A consequence of this latter bound is the first decoupling,%
\[
\int_{\mathbb{R}^{2}}\left\vert \sum_{m,n=1}^{2^{s}}G_{m}\ast G_{n}\left(
z\right)  \right\vert ^{p}dz\leq C_{p}\sum_{m,n=1}^{2^{s}}\int_{\mathbb{R}%
^{2}}\left\vert G_{m}\ast G_{n}\left(  z\right)  \right\vert ^{p}%
dz,\ \ \ \ \ \text{for }1\leq p<\infty.
\]
In particular, with $p=\frac{q}{q-2}$ we obtain%
\[
\int_{B\left(  0,2^{2s}\right)  }\left\vert \mathcal{E}F\left(  \xi\right)
\right\vert ^{q}d\xi\lesssim\left(  \int_{\mathbb{R}^{2}}\left\vert
\sum_{m,n=1}^{2^{s}}G_{m}\ast G_{n}\left(  z\right)  \right\vert ^{\frac
{q}{q-2}}dz\right)  ^{\frac{q-2}{2}}\lesssim\left(  \sum_{m,n=1}^{2^{s}}%
\int_{\mathbb{R}^{2}}\left\vert G_{m}\ast G_{n}\left(  z\right)  \right\vert
^{\frac{q}{q-2}}dz\right)  ^{\frac{q-2}{2}}.
\]

Now we bound $G_{k}$ by
\begin{align*}
\left\vert G_{k}\right\vert  &  \leq\left\Vert G_{k}\right\Vert _{L^{\infty}%
}\mathbf{1}_{\operatorname*{Supp}g_{k}}\leq\left\Vert \varphi_{s}\ast
\Phi_{\ast}F_{k}\right\Vert _{L^{\infty}}\mathbf{1}_{R_{k}}=\left(  \sup_{z\in
R_{k}}\left\vert \int2^{4s}\varphi\left(  \frac{z-w}{2^{-2s}}\right)
d\Phi_{\ast}F_{k}\left(  w\right)  \right\vert \right)  \mathbf{1}_{R_{k}}\\
&  \leq2^{4s}2^{-2s}\left\vert a_{k}\right\vert \mathbf{1}_{R_{k}}%
=2^{2s}\left\vert a_{k}\right\vert \mathbf{1}_{R_{k}}\ ,
\end{align*}
which implies%
\begin{align*}
\int_{\mathbb{R}^{2}}\left\vert G_{m}\ast G_{n}\left(  z\right)  \right\vert
^{\frac{q}{q-2}}dz  &  \lesssim\int_{\mathbb{R}^{2}}\left\vert 2^{2s}%
\left\vert a_{m}\right\vert \mathbf{1}_{R_{m}}\ast2^{2s}\left\vert
a_{n}\right\vert \mathbf{1}_{R_{n}}\left(  z\right)  \right\vert ^{\frac
{q}{q-2}}dz\\
&  \lesssim\left\vert a_{m}\right\vert ^{\frac{q}{q-2}}\left\vert
a_{n}\right\vert ^{\frac{q}{q-2}}2^{\frac{4q}{q-2}s}\int_{\mathbb{R}^{2}%
}\left\vert \mathbf{1}_{R_{m}}\ast\mathbf{1}_{R_{n}}\left(  z\right)
\right\vert ^{\frac{q}{q-2}}dz.
\end{align*}
If we denote by $\theta_{m,n}$ the angle between the vectors perpendicular to
$\Phi\left(  I_{m}\right)  $ and $\Phi\left(  I_{n}\right)  $, then we have
$\theta_{m,n}\approx2^{-s}+\left\vert m-n\right\vert $ and so%
\begin{align*}
&  \left\vert \mathbf{1}_{R_{m}}\ast\mathbf{1}_{R_{n}}\left(  z\right)
\right\vert \leq\left\vert R_{m}\cap R_{n}\right\vert \mathbf{1}_{R_{m}+R_{n}%
}\left(  z\right)  \lesssim\frac{2^{-4s}}{2^{-s}+\left\vert \theta
_{m,n}\right\vert }\mathbf{1}_{R_{m}+R_{n}}\left(  z\right)  =\frac{2^{-3s}%
}{1+\left\vert m-n\right\vert }\mathbf{1}_{R_{m}+R_{n}}\left(  z\right)  ,\\
&  \ \ \ \ \ \ \ \ \ \ \ \ \ \ \ \text{and }\left\vert R_{m}+R_{n}\right\vert
\approx2^{-s}2^{-s}\left(  2^{-s}+\left\vert \theta_{m,n}\right\vert \right)
\approx\left(  1+\left\vert m-n\right\vert \right)  2^{-3s}.
\end{align*}

As a consequence we conclude that for any function $F$ on $I_{0}$ as in
(\ref{constant}), that is constant on each interval $I\in\mathcal{G}%
_{s}\left[  I_{0}\right]  $,%
\begin{align*}
&  \left(  \int_{B\left(  0,2^{2s}\right)  }\left\vert \mathcal{E}F\left(
\xi\right)  \right\vert ^{q}d\xi\right)  ^{\frac{2}{q-2}}\lesssim\sum
_{m,n=1}^{2^{s}}\left\vert a_{m}\right\vert ^{\frac{q}{q-2}}\left\vert
a_{n}\right\vert ^{\frac{q}{q-2}}2^{\frac{4q}{q-2}s}\int_{\mathbb{R}^{2}%
}\left\vert \mathbf{1}_{R_{m}}\ast\mathbf{1}_{R_{n}}\left(  z\right)
\right\vert ^{\frac{q}{q-2}}dz\\
&  \lesssim\sum_{m,n=1}^{2^{s}}\left\vert a_{m}\right\vert ^{\frac{q}{q-2}%
}\left\vert a_{n}\right\vert ^{\frac{q}{q-2}}2^{\frac{4q}{q-2}s}\left(
\frac{2^{-3s}}{1+\left\vert m-n\right\vert }\right)  ^{\frac{q}{q-2}}\left(
1+\left\vert m-n\right\vert \right)  2^{-3s}\\
&  =\sum_{m,n=1}^{2^{s}}\left\vert a_{m}\right\vert ^{\frac{q}{q-2}}\left\vert
a_{n}\right\vert ^{\frac{q}{q-2}}\frac{2^{\frac{4q}{q-2}s}2^{-\frac{3q}{q-2}%
s}2^{-3s}}{\left(  1+\left\vert m-n\right\vert \right)  ^{\frac{q}{q-2}-1}%
}\leq\sum_{m,n=1}^{2^{s}}\frac{\left\vert a_{m}\right\vert ^{\frac{2q}{q-2}%
}+\left\vert a_{n}\right\vert ^{\frac{2q}{q-2}}}{2}\frac{2^{-\frac{2q-6}%
{q-2}s}}{\left(  1+\left\vert m-n\right\vert \right)  ^{\frac{2}{q-2}}},
\end{align*}
which by symmetry of the double sum equals%
\begin{align*}
&  2^{-\frac{2q-6}{q-2}s}\sum_{n=1}^{2^{s}}\left\vert a_{n}\right\vert
^{\frac{2q}{q-2}}\sum_{m=1}^{2^{s}}\frac{1}{\left(  1+\left\vert
m-n\right\vert \right)  ^{\frac{2}{q-2}}}\approx2^{-\frac{2q-6}{q-2}s}%
\sum_{n=1}^{2^{s}}\left\vert a_{n}\right\vert ^{\frac{2q}{q-2}}2^{\frac
{q-4}{q-2}s}=2^{-s}\sum_{n=1}^{2^{s}}\left\vert a_{n}\right\vert ^{\frac
{2q}{q-2}}\\
&  =2^{-s}\sum_{n=1}^{2^{s}}\frac{1}{\left\vert I_{n}\right\vert }\left\Vert
F_{n}\right\Vert _{L^{\frac{2q}{q-2}}\left(  I_{0}\right)  }^{\frac{2q}{q-2}%
}=\sum_{n=1}^{2^{s}}\left\Vert F_{n}\right\Vert _{L^{\frac{2q}{q-2}}\left(
I_{0}\right)  }^{\frac{2q}{q-2}}=\left\Vert \sum_{n=1}^{2^{s}}F_{n}\right\Vert
_{L^{\frac{2q}{q-2}}\left(  I_{0}\right)  }^{\frac{2q}{q-2}}=\left\Vert
F\right\Vert _{L^{\frac{2q}{q-2}}\left(  I_{0}\right)  }^{\frac{2q}{q-2}},
\end{align*}
which constitutes the second decoupling, and where the implied constants are
independent of $s\in\mathbb{N}$.

Taking $F=\mathsf{Q}_{t}f$ for any $1\leq t\leq s$, we have proved%
\[
\left\Vert \mathcal{E}\mathsf{Q}_{t}f\right\Vert _{L^{q}\left(  B\left(
0,2^{2s}\right)  \right)  }^{q}\lesssim\left(  \left\Vert \mathsf{Q}%
_{t}f\right\Vert _{L^{\frac{2q}{q-2}}\left(  I_{0}\right)  }^{\frac{2q}{q-2}%
}\right)  ^{\frac{q-2}{2}}\leq\left\Vert f\right\Vert _{L^{\frac{2q}{q-2}%
}\left(  I_{0}\right)  }^{q}\lesssim\left\Vert f\right\Vert _{L^{q}\left(
I_{0}\right)  }^{q},
\]
since $\frac{2q}{q-2}<q\ $if $q>4$. Letting first $s\nearrow\infty$, we obtain%
\[
\int_{\mathbb{R}^{2}}\left\vert \mathcal{E}\mathsf{Q}_{t}f\left(  \xi\right)
\right\vert ^{q}d\xi\lesssim\left\Vert f\right\Vert _{L^{q}\left(
I_{0}\right)  }^{q},\ \ \ \ \ \text{for all }t\in\mathbb{N},
\]
and then letting $t\nearrow\infty$, we obtain from Fatou's lemma that for
$q>4$,%
\[
\int_{\mathbb{R}^{2}}\left\vert \mathcal{E}f\left(  \xi\right)  \right\vert
^{q}d\xi=\int_{\mathbb{R}^{2}}\lim\inf_{t\rightarrow\infty}\ \left\vert
\mathcal{E}\mathsf{Q}_{t}f\left(  \xi\right)  \right\vert ^{q}d\xi\leq\lim
\inf_{t\rightarrow\infty}\ \int_{\mathbb{R}^{2}}\left\vert \mathcal{E}%
\mathsf{Q}_{t}f\left(  \xi\right)  \right\vert ^{q}d\xi\lesssim\left\Vert
f\right\Vert _{L^{q}\left(  I_{0}\right)  }^{q},
\]
since $\mathsf{Q}_{t}f\rightarrow f$ in $L^{2}\left(  I_{0}\right)  $ implies
that for every $\xi\in\mathbb{R}^{2}$,
\[
\lim_{t\rightarrow\infty}\mathcal{E}\mathsf{Q}_{t}f\left(  \xi\right)
=\lim_{t\rightarrow\infty}\int_{0}^{1}e^{-i\left(  \xi_{1}x+\xi_{2}%
x^{2}\right)  }\mathsf{Q}_{t}f\left(  x\right)  dx=\int_{0}^{1}e^{-i\left(
\xi_{1}x+\xi_{2}x^{2}\right)  }f\left(  x\right)  dx=\mathcal{E}f\left(
\xi\right)  .
\]

\section{Bilinear wave packet induction on scales argument}

Fefferman's beautiful convolution decoupling (\ref{Feff over}) fails in higher
dimensions, as is easily seen in three dimensions by rotating the parabola
into a paraboloid and noting that convolutions of pairs of antipodal
$3D$-rectangles lying above a circle in the horizontal plane, have essentially
the same support. Moreover, the critical index in dimension $d\geq3$ is
$\frac{2d}{d-1}$, which is no longer an even integer, thus preventing the use
of Hausdorff-Young to implement the above method.

We now present a second argument in five steps, in which Fefferman's
convolution decoupling is essentially circumvented by the wave packet
expansion, that decomposes an individual piece of the Fourier extension into
an $\ell^{2}$-weighted sum of translates and results in the one-to-one
property used at the end of \textbf{Step 4}, but the fact that $4$ is an even
integer still plays a role there:

\begin{enumerate}
\item Reduce matters to a local bilinear testing inequality.

\item Expand the functions into their wave packet decompositions.

\item Extend discrete characterizations of Fourier extension to the bilinear setting.

\item Derive bilinear wave packet estimates

\item Iterate the estimates on scales
\end{enumerate}

\subsection{Step 1: bilinear testing condition}

It is well known that using Tao's $\varepsilon$-removal theorem in
\cite{Tao1}, and Nikishin-Maurey-Pisier theory, see e.g. Bourgain \cite{Bou}
for the circle and comments on the parabola, and Buschenhenke \cite{Bus} for
complete details on the parabola, the Fourier extension theorem in the plane
is equivalent to the local\ bilinear inequality (\cite{TaVaVe}, see also
\cite{BoGu}),%
\begin{align}
&  \left\Vert \mathcal{E}f_{1}\mathcal{E}f_{2}\right\Vert _{L^{2}\left(
B\left(  0,R\right)  \right)  }\leq C_{\varepsilon}R^{\varepsilon}\left\Vert
f_{1}\right\Vert _{L^{\infty}\left(  U_{1}\right)  }\left\Vert f_{2}%
\right\Vert _{L^{\infty}\left(  U_{2}\right)  },\label{bi in}\\
&  \ \ \ \ \ \ \ \ \ \ \text{for all }f_{\mu}\text{ bounded with support in
}U_{\mu}\ ,\nonumber
\end{align}
taken over $\nu$-separated intervals $U_{1}$ and $U_{2}$, for some
sufficiently small $\nu>0$. We will take this equivalence as a black box here
(it holds in all dimensions).

\subsection{Step 2: wave packet decomposition}

Recall that $\Gamma_{t}=2^{-t}\mathbb{Z}_{2^{t}}$ for $t\in\mathbb{N}$. Now
fix $s\in2\mathbb{N}$ and write the tiling $\mathcal{D}_{\frac{1}{2}s}\left[
U\right]  =\left\{  I_{\frac{1}{2}s}+n\right\}  _{n\in\Gamma_{\frac{1}{2}s}}$
as%
\[
\mathcal{D}_{\frac{1}{2}s}\left[  U\right]  =\mathcal{D}_{\frac{1}{2}%
s}^{\operatorname*{even}}\left[  U\right]  \bigcup\mathcal{D}_{\frac{1}{2}%
s}^{\operatorname*{odd}}\left[  U\right]  ,
\]
where $2^{\frac{1}{2}s}n\in\mathbb{Z}_{2^{\frac{1}{2}s}}$ is even and odd
respectively in $\left\{  I_{\frac{1}{2}s}+n\right\}  _{n\in\Gamma_{\frac
{1}{2}s}}$. By symmetry it suffices to consider just the case $\mathcal{G}%
_{\frac{1}{2}s}^{\operatorname*{even}}\left[  U\right]  $, and to reduce
clutter of notation, we will suppress the superscript $\operatorname*{even}$,
and just write $\mathcal{G}_{\frac{1}{2}s}\left[  U\right]  $, i.e.%
\begin{equation}
\mathcal{G}_{\frac{1}{2}s}\left[  U\right]  =\mathcal{G}_{\frac{1}{2}%
s}^{\operatorname*{even}}\left[  U\right]  .\label{doubles}%
\end{equation}

Following the wave packet decomposition\ as in \cite[Chapter 2]{Dem} (which
also holds in all dimensions), we apply the torus Fourier series decomposition
to the function $f_{L}\equiv\mathbf{1}_{L}f\ $in the doubled interval $2L$
viewed as a torus, where these doubled intervals are pairwise disjoint because
of our assumption (\ref{doubles}).

\begin{notation}
In order to simplify notation, we will instead assume $f_{L}\equiv
\mathbf{1}_{\frac{1}{2}L}f$, and that $\chi_{L}$ is $1$ on $\frac{1}{2}L$ and
supported in $L$, and then apply the torus Fourier series decomposition of
$f_{L}$ on the torus $L$.
\end{notation}

\begin{notation}
For any interval $I$, we define $\lambda_{I}$ to be its left hand endpoint.
\end{notation}

Using the orthonormal basis $\left\{  \frac{e^{2\pi i\lambda_{K}x}}%
{\sqrt{\left\vert L\right\vert }}\right\}  _{K\in\mathcal{G}_{-\frac{1}{2}s}}$
of $L^{2}\left(  L\right)  $, and including smooth cutoff functions $\chi_{L}$
that are $1$ on $L$ and supported in $2L$, the result is, with a slight abuse
of notation taking $2L$ back to $L$,%
\[
f\left(  x\right)  =\sum_{L\in\mathcal{G}_{\frac{1}{2}s}\left[  U\right]
}\sum_{K\in\mathcal{G}_{-\frac{1}{2}s}}\widehat{f_{L}}\left(  K\right)
\chi_{L}\left(  x\right)  \frac{e^{2\pi i\lambda_{K}x}}{\sqrt{\left\vert
L\right\vert }}=\sum_{L\in\mathcal{G}_{\frac{1}{2}s}\left[  U\right]  }%
\sum_{K\in\mathcal{G}_{-\frac{1}{2}s}}e^{2\pi i\lambda_{K}\lambda_{L}%
}\widetilde{f_{L}}\left(  K\right)  \chi_{L}\left(  x\right)  e^{2\pi
i\lambda_{K}\left(  x-\lambda_{L}\right)  },
\]
where%
\begin{equation}
\widehat{f_{L}}\left(  K\right)  \equiv\int f_{L}\left(  y\right)
\frac{e^{-2\pi i\lambda_{K}y}}{\sqrt{\left\vert L\right\vert }}dy\text{ and
}\widetilde{f_{L}}\left(  K\right)  =\frac{\widehat{f_{L}}\left(  K\right)
}{\sqrt{\left\vert L\right\vert }}=R^{\frac{1}{4}}\widehat{f_{L}}\left(
K\right)  .\label{conn}%
\end{equation}
Finally we take the Fourier extension of $f$ to obtain the wave packet
decomposition,%
\[
\mathcal{E}f\left(  \xi\right)  =\sum_{L\in\mathcal{G}_{\frac{1}{2}s}\left[
U\right]  }\sum_{K\in\mathcal{G}_{-\frac{1}{2}s}}e^{2\pi i\lambda_{K}%
\lambda_{L}}\widetilde{f_{L}}\left(  K\right)  \mathcal{E}\left(
\tau_{-\lambda_{L}}\exp_{2\pi i\lambda_{K}}\chi_{L}\right)  \left(
\xi\right)  .
\]

Now we compute that%
\begin{align*}
\mathcal{E}\left(  \tau_{-\lambda_{L}}\exp_{2\pi i\lambda}\chi_{L}\right)
\left(  \xi\right)   &  =\left[  \Phi_{\ast}\left(  \tau_{-\lambda_{L}}%
\exp_{2\pi i\lambda}\chi_{L}\right)  \right]  ^{\wedge}\left(  \xi\right)
=\int e^{-2\pi i\xi\cdot\Phi\left(  x\right)  }e^{2\pi i\lambda\left(
x-\lambda_{L}\right)  }\chi_{L}\left(  x\right)  dx\\
&  =e^{-2\pi i\lambda\lambda_{L}}\int e^{-2\pi i\left[  \left(  \xi
_{1}-\lambda\right)  x+\xi_{2}x^{2}\right]  }\chi_{L}\left(  x\right)
dx=e^{-2\pi i\lambda\lambda_{L}}\mathcal{E}\chi_{L}\left(  \xi-\left(
\lambda,0\right)  \right)  ,
\end{align*}
i.e. that $\mathcal{E}\left(  \tau_{-\lambda_{L}}\exp_{i\lambda}\chi
_{L}\right)  $ is $e^{-2\pi i\lambda\lambda_{L}}$ times $\mathcal{E}\left(
\exp_{2\pi i\lambda}\chi_{L}\right)  $, which is a \emph{horizontal} translate
of $\mathcal{E}\chi_{L}$ by $\left(  \lambda,0\right)  $. Moreover, since we
are ultimately integrating in $\xi$ over the cube $Q\left(  0,2^{s}\right)  $,
we can without loss of generality multiply the Fourier extension by
$\widehat{\phi_{2^{s}}}$, i.e. convolve the singular measure $\Phi_{\ast
}\left(  \exp_{2\pi i\lambda}\chi_{L}\right)  $ with an appropriate
approximate identity $\phi_{2^{s}}$ at scale $2^{-s}$, and obtain%
\begin{align*}
&  \ \ \ \ \ \ \ \ \ \ \ \ \ \ \ \widehat{\phi_{2^{s}}}\left(  \xi\right)
\mathcal{E}_{S}\left(  \exp_{2\pi i\lambda}\chi_{L}\right)  \left(
\xi\right)  =\left[  \phi_{2^{s}}\ast\Phi_{\ast}\left(  \exp_{2\pi i\lambda
}\chi_{L}\right)  \right]  ^{\wedge}\left(  \xi\right) \\
&  =\int e^{-2\pi i\xi z}\left(  \int\Phi_{\ast}\left(  \exp_{2\pi i\lambda
}\chi_{L}\right)  \left(  z-w\right)  \phi_{2^{s}}\left(  w\right)  dw\right)
dz=\int\left(  \int e^{-2\pi i\xi z}\Phi_{\ast}\left(  \exp_{2\pi i\lambda
}\chi_{L}\right)  \left(  z-w\right)  dz\right)  \phi_{2^{s}}\left(  w\right)
dw\\
&  =\int\left(  \int e^{-2\pi i\xi z}\left[  \tau_{-w}\Phi_{\ast}\left(
\exp_{2\pi i\lambda}\chi_{L}\right)  \right]  \left(  z\right)  dz\right)
\phi_{2^{s}}\left(  w\right)  dw=\int\left(  \int e^{-2\pi i\xi z}\Phi_{\ast
}^{w}\left(  \exp_{2\pi i\lambda}\chi_{L}\right)  \left(  z\right)  dz\right)
\phi_{2^{s}}\left(  w\right)  dw
\end{align*}
which gives%
\begin{align}
&  \widehat{\phi_{2^{s}}}\left(  \xi\right)  \mathcal{E}_{S}\left(  \exp_{2\pi
i\lambda}\chi_{L}\right)  \left(  \xi\right)  =\int\left(  \int e^{-2\pi i\xi
z}\Phi_{\ast}^{w}\left(  \exp_{2\pi i\lambda}\chi_{L}\right)  \left(
z\right)  dz\right)  \phi_{2^{s}}\left(  w\right)  dw\label{horiz}\\
&  =\int\left[  \Phi_{\ast}^{w}\left(  \exp_{2\pi i\lambda}\chi_{L}\right)
\right]  ^{\wedge}\left(  \xi\right)  \phi_{2^{s}}\left(  w\right)
dw=\int\mathcal{E}_{\tau_{w}S}\left(  \exp_{2\pi i\lambda}\chi_{L}\right)
\left(  \xi\right)  \phi_{2^{s}}\left(  w\right)  dw\nonumber\\
&  =\int\left\{  \mathcal{E}_{\tau_{w}S}\chi_{L}\left(  \xi-\left(
\lambda,0\right)  \right)  \right\}  \phi_{2^{s}}\left(  w\right)
dw,\nonumber
\end{align}
which is a horizontal translate of the function $\int\left\{  \mathcal{E}%
_{\tau_{w}S}\chi_{L}\left(  \xi\right)  \right\}  \phi_{2^{s}}\left(
w\right)  dw $, where $\tau_{w}S$ is the curve $S$ translated by $w$.
Moreover, we may further multiply both sides of (\ref{horiz}) by
$\widehat{\phi_{2^{s}}}\left(  \xi-\left(  \lambda,0\right)  \right)  $ if
$\left\vert \lambda\right\vert \leq2^{s-1}=\frac{1}{2}R$, provided we are
integrating over $Q\left(  0,\frac{1}{2}2^{s}\right)  $, which we may assume
without loss of generality, and we will return to this point momentarily.

Now $\phi_{2^{s}}\ast\Phi_{\ast}\chi_{L}$ is smoothly adapted to the
translated plate $P_{L}$ given by%
\[
P_{L}=\Phi\left(  \lambda_{L}\right)  +V_{L}\ ,
\]
where $V_{L}$ is the rotated plate with dimensions $R^{-\frac{1}{2}}\times
R^{-1}$ and bottom left vertex at the origin, and having normal to the large
face equal to the unit normal $\mathbf{n}_{\Phi\left(  \lambda_{L}\right)
}^{S}$ to the surface $S$ at the point $\Phi\left(  \lambda_{L}\right)  $.
Moreover, the $L^{\infty}$ norm of the convolution $\phi_{2^{s}}\ast\Phi
_{\ast}\chi$ is%
\[
\left\Vert \phi_{2^{s}}\ast\Phi_{\ast}\chi_{L}\right\Vert _{L^{\infty}}%
\approx\frac{\left\Vert \phi_{2^{s}}\ast\Phi_{\ast}\chi_{L}\right\Vert
_{L^{1}}}{\left\vert \operatorname*{Supp}\left(  \phi_{2^{s}}\ast\Phi_{\ast
}\chi_{L}\right)  \right\vert }\approx\frac{\left\vert L\right\vert
}{R^{-\frac{1}{2}}\times R^{-1}}=\frac{R^{-\frac{1}{2}}}{R^{-\frac{1}{2}%
}\times R^{-1}}=R.
\]
Then the function given by the modulation $e^{-i\xi\cdot\Phi\left(
c_{L}\right)  }$ times the Fourier transform of $\phi_{2^{s}}\ast\Phi_{\ast
}\chi_{L}$, namely
\[
e^{-i\xi\cdot\Phi\left(  c_{L}\right)  }\left(  \phi_{2^{s}}\ast\Phi_{\ast
}\chi_{L}\right)  ^{\wedge}\left(  \xi\right)  =e^{-i\xi\cdot\Phi\left(
c_{L}\right)  }\widehat{\phi_{2^{s}}}\left(  \xi\right)  \mathcal{E}_{S}%
\chi_{L}\left(  \xi\right)  ,
\]
is smoothly adapted to the rectangular tube $T_{L}$ that is defined to be
centered at the origin with dimensions $R^{\frac{1}{2}}\times R$, and having
long edge parallel to $\mathbf{n}_{\Phi\left(  \lambda_{L}\right)  }^{S}$.
Moreover, the $L^{\infty}$ norm of this function is roughly $\left\vert
L\right\vert =R^{-\frac{1}{2}}$.

Since the translated curves $\tau_{w}S$ satisfy the same estimates as $S$, we
conclude that the corresponding functions
\[
\Lambda_{w}^{L}\left(  \xi\right)  \equiv e^{-2\pi i\xi\cdot\left(
w+\Phi\left(  c_{L}\right)  \right)  }\widehat{\phi_{2^{s}}}\left(
\xi\right)  \mathcal{E}_{\tau_{w}S}\chi_{L}\left(  \xi\right)
\]
are uniformly smoothly adapted to the tube $T_{L}$ with $L^{\infty}$ norm
$R^{-\frac{1}{2}}$, since a small perturbation of $S$ results in a
perturbation of $T_{L}$ that is negligible for the purposes of being smoothly
adapted. Moreover, we also have that the functions $\exp_{iw}\Lambda_{w}^{L}$
are uniformly smoothly adapted to the tube $T_{L}$ with $L^{\infty}$ norm
$R^{-\frac{1}{2}}$ for similar reasons using $\left\vert w\right\vert
\leq2^{-s}$.

Now we set $\xi_{\lambda}\equiv\xi-\left(  \lambda,0\right)  $ for notational
convenience, and note that%
\[
\widehat{\phi_{2^{s}}}\left(  \xi\right)  \mathcal{E}_{\tau_{w}S}\chi
_{L}\left(  \xi\right)  =e^{2\pi i\xi\cdot\left(  w+\Phi\left(  c_{L}\right)
\right)  }\Lambda_{w}^{L}\left(  \xi\right)  =e^{2\pi i\xi\cdot\Phi\left(
c_{L}\right)  }\left(  \exp_{2\pi iw}\Lambda_{w}^{L}\right)  \left(
\xi\right)  .
\]
So if we multiply both sides of (\ref{horiz}) by $\widehat{\phi_{2^{s}}%
}\left(  \xi_{\lambda}\right)  $, we obtain that%
\begin{align*}
&  \widehat{\phi_{2^{s}}}\left(  \xi_{\lambda}\right)  \widehat{\phi_{2^{s}}%
}\left(  \xi\right)  \mathcal{E}_{S}\left(  \exp_{2\pi i\lambda}\chi
_{L}\right)  \left(  \xi\right)  =\widehat{\phi_{2^{s}}}\left(  \xi_{\lambda
}\right)  \int\left\{  \mathcal{E}_{\tau_{w}S}\chi_{L}\left(  \xi_{\lambda
}\right)  \right\}  \phi_{2^{s}}\left(  w\right)  dw\\
&  =\int\left\{  \widehat{\phi_{2^{s}}}\left(  \xi_{\lambda}\right)
\mathcal{E}_{\tau_{w}S}\chi_{L}\left(  \xi_{\lambda}\right)  \right\}
\phi_{2^{s}}\left(  w\right)  dw=\int\left\{  e^{i\xi_{\lambda}\cdot
\Phi\left(  c_{L}\right)  }\left(  \exp_{2\pi iw}\Lambda_{w}^{L}\right)
\left(  \xi_{\lambda}\right)  \right\}  \phi_{2^{s}}\left(  w\right)  dw\\
&  =e^{i\xi_{\lambda}\cdot\Phi\left(  c_{L}\right)  }\left\{  \int\left(
\exp_{2\pi iw}\Lambda_{w}^{L}\right)  \phi_{2^{s}}\left(  w\right)
dw\right\}  \left(  \xi_{\lambda}\right)  =e^{2\pi i\xi_{\lambda}\cdot
\Phi\left(  c_{L}\right)  }\Lambda^{L}\left(  \xi_{\lambda}\right)  =\left(
\exp_{2\pi i\Phi\left(  c_{L}\right)  }\Lambda^{L}\right)  \left(
\xi_{\lambda}\right)  ,
\end{align*}
where $\Lambda^{L}$ is an average $\int\cdot\phi_{2^{s}}\left(  w\right)  dw$
of functions $\exp_{2\pi iw}\Lambda_{w}^{L}$ that are uniformly smoothly
adapted to the tube $T_{L}$ with $L^{\infty}$ norm $R^{-\frac{1}{2}}$, and
hence $\Lambda^{L}$ itself is smoothly adapted to the tube $T_{L}$ with
$L^{\infty}$ norm $R^{-\frac{1}{2}}$.

Altogether then, we see that $\widehat{\phi_{2^{s}}}\left(  \xi_{\lambda
}\right)  \widehat{\phi_{2^{s}}}\left(  \xi\right)  \mathcal{E}_{S}\left(
\exp_{2\pi i\lambda}\chi_{L}\right)  \left(  \xi\right)  $ is a horizontal
translate by $\left(  \lambda,0\right)  $ of the function $\exp_{2\pi
i\Phi\left(  c_{L}\right)  }\Lambda^{L}$, which is a modulation $\exp_{2\pi
i\Phi\left(  c_{L}\right)  }$ times a function $\Lambda^{L}$ that is smoothly
adapted to the tube $T_{L}$ with $L^{\infty}$ norm $R^{-\frac{1}{2}}$. Thus in
calculating the norm on the left hand side of the bilinear inequality
(\ref{bi in}), it is enough to replace the function $\mathcal{E}_{S}\left(
\exp_{2\pi i\lambda}\chi_{L}\right)  \left(  \xi\right)  $ with the function
$\left(  \exp_{2\pi i\Phi\left(  c_{L}\right)  }\Lambda^{L}\right)  \left(
\xi-\left(  \lambda,0\right)  \right)  $, and furthermore, by the rapid decay
of derivatives away from the tube, matters can then be reduced in a standard
way to replacing $\mathcal{E}_{S}\left(  \exp_{2\pi i\lambda}\chi_{L}\right)
\left(  \xi\right)  $ with just the function
\[
R^{-\frac{1}{2}}\left(  \exp_{2\pi i\Phi\left(  c_{L}\right)  }\mathbf{1}%
_{T_{L}}\right)  \left(  \xi-\left(  \lambda,0\right)  \right)  =R^{-\frac
{1}{2}}e^{2\pi i\left(  \xi-\left(  \lambda,0\right)  \right)  \cdot
\Phi\left(  c_{L}\right)  }\mathbf{1}_{T_{L}}\left(  \xi-\left(
\lambda,0\right)  \right)  .
\]

\begin{conclusion}
\label{gather}For $\lambda=\lambda_{K}$, which satisfies $\left\vert
\lambda\right\vert \leq R$, the function $\widehat{\phi_{2^{s}}}\left(
\xi_{\lambda}\right)  \widehat{\phi_{2^{s}}}\left(  \xi\right)  \mathcal{E}%
\left(  \exp_{2\pi i\lambda}\chi_{L}\right)  \left(  \xi\right)  $ is a
modulation $e^{2\pi i\left(  \xi-\left(  \lambda_{K},0\right)  \right)
\cdot\Phi\left(  c_{L}\right)  }$ times a function that is smoothly adapted to
the tube $T_{K,L}\equiv\tau_{-\left(  \lambda_{K},0\right)  }T_{L}%
=T_{L}-\left(  \lambda_{K},0\right)  $, and thus it suffices by rapid decay,
to replace the function $\mathcal{E}\left(  \exp_{i\lambda}\chi_{L}\right)
\left(  \xi\right)  $ in our subsequent norm estimates with the function%
\[
e^{2\pi i\left(  \xi-\left(  \lambda_{K},0\right)  \right)  \cdot\Phi\left(
c_{L}\right)  }\mathbf{1}_{T_{K,L}}\left(  \xi\right)  =R^{-\frac{1}{2}%
}e^{2\pi i\xi\cdot\Phi\left(  c_{L}\right)  }e^{-i\lambda_{K}\cdot\lambda_{L}%
}\mathbf{1}_{T_{K,L}}\left(  \xi\right)  .
\]
Moreover, for intervals $K\notin\mathcal{G}_{-\frac{1}{2}s}\left[  \left[
-2^{s},2^{s}\right)  \right]  $, the tubes $T_{K,L}$ lie outside the square
$Q\left(  0,2^{s}\right)  $, and so may be discarded by rapid decay, since the
norm estimates in the bilinear inequality (\ref{bi in}) are restricted to
$Q\left(  0,2^{s}\right)  $. One can picture the above wave packet
decomposition in the plane as follows. For each dyadic interval $L\in
\mathcal{G}_{\frac{1}{2}s}\left[  \left[  -1,1\right)  \right]  $ at scale
$2^{-\frac{1}{2}s}$ that is contained the space side interval $\left[
-1,1\right)  $, and for each dyadic interval $K\in\mathcal{G}_{-\frac{1}{2}%
s}\left[  \left[  -2^{s},2^{s}\right)  \right]  $ at scale $2^{\frac{1}{2}s}$
that is contained in the Fourier series side interval $\left[  -2^{s}%
,2^{s}\right)  $, the Fourier extension $\mathcal{E}\chi_{K,L}\left(
\xi\right)  $ of the function $\chi_{K,L}$, is a modulation times a function
that is smoothly adapted to the tube $T_{K,L}=\tau_{\lambda_{K}}T_{L}$, which
is the translate by $\lambda_{K}=kR^{\frac{1}{2}}$ of the tube $T_{L}$ that is
centered at the origin with dimensions $R^{\frac{1}{2}}\times R$ with long
edge parallel to the unit normal $\mathbf{n}_{\Phi\left(  \lambda_{L}\right)
}^{S}$ to the surface $S$ at the point $\Phi\left(  \lambda_{L}\right)  $. The
unit normals are separated by $R^{-\frac{1}{2}}$, i.e. $\left\vert
\mathbf{n}_{\Phi\left(  \lambda_{L}\right)  }^{S}-\mathbf{n}_{\Phi\left(
\lambda_{L^{\prime}}\right)  }^{S}\right\vert \geq R^{-\frac{1}{2}}$ for
$L,L^{\prime}\in\mathcal{G}_{\frac{1}{2}s}\left[  \left[  -1,1\right)
\right]  $ with $L\neq L^{\prime}$.
\end{conclusion}

Restricting $K$\ to $\mathcal{G}_{-\frac{1}{2}s}\left[  \left[  -2^{s}%
,2^{s}\right)  \right]  $ as in Conclusion \ref{gather}, we may write the wave
packet decomposition of each $f_{\mu}$ as%
\begin{align}
\mathcal{E}f_{\mu}\left(  \xi\right)   &  =\sum_{\substack{K_{\mu}%
\in\mathcal{G}_{-\frac{1}{2}s}\left[  \left[  -2^{s},2^{s}\right)  \right]
\\L_{\mu}\in\mathcal{G}_{\frac{1}{2}s}\left[  U_{\mu}\right]  }}e^{2\pi
i\lambda_{K_{\mu}}\lambda_{L_{\mu}}}\widetilde{f_{L_{\mu}}}\left(
\lambda_{K_{\mu}}\right)  \mathcal{E}\varphi_{K_{\mu},L_{\mu}}\left(
\xi\right) \label{wp}\\
&  \sim R^{-\frac{1}{2}}\sum_{\substack{K_{\mu}\in\mathcal{G}_{-\frac{1}{2}%
s}\left[  \left[  -2^{s},2^{s}\right)  \right]  \\L_{\mu}\in\mathcal{G}%
_{\frac{1}{2}s}\left[  U_{\mu}\right]  }}e^{2\pi i\lambda_{K_{\mu}}%
\lambda_{L_{\mu}}}\widetilde{f_{L_{\mu}}}\left(  \lambda_{K_{\mu}}\right)
e^{-2\pi i\xi\cdot\Phi\left(  \lambda_{L_{\mu}}\right)  }e^{-2\pi
i\lambda_{K_{\mu}}\lambda_{L_{\mu}}}\mathbf{1}_{T_{K_{\mu},L_{\mu}}}\left(
\xi\right) \nonumber\\
&  =R^{-\frac{1}{2}}\sum_{\substack{K_{\mu}\in\mathcal{G}_{-\frac{1}{2}%
s}\left[  \left[  -2^{s},2^{s}\right)  \right]  \\L_{\mu}\in\mathcal{G}%
_{\frac{1}{2}s}\left[  U_{\mu}\right]  }}e^{-2\pi i\xi\cdot\Phi\left(
\lambda_{L_{\mu}}\right)  }\widetilde{f_{L_{\mu}}}\left(  \lambda_{K_{\mu}%
}\right)  \mathbf{1}_{T_{K_{\mu},L_{\mu}}}\left(  \xi\right) \nonumber\\
&  =R^{-\frac{1}{4}}\sum_{\alpha}\sum_{\left(  K_{\mu},L_{\mu}\right)
\sim\alpha}e^{-2\pi i\xi\cdot\Phi\left(  \lambda_{L_{\mu}}\right)  }%
\widehat{f_{L_{\mu}}}\left(  \lambda_{K_{\mu}}\right)  \mathbf{1}_{T_{K_{\mu
},L_{\mu}}}\left(  \xi\right)  ,\nonumber
\end{align}
where the sum ranges over a collection of pairwise disjoint squares $\alpha$
of side length $R^{\frac{1}{2}}$ that cover the top half of $Q\left(
0,R\right)  $, and where for a fixed square $\alpha$, we write $\left(
K,L\right)  \sim\alpha$ to mean $T_{K,L}\cap\alpha\neq\emptyset$, and it
remains to estimate the norm of the product function $\prod_{\mu=1}%
^{2}\mathcal{E}f_{\mu}$ using the discrete characterization of Fourier extension.

\subsection{Step 3: discrete characterizations of Fourier extension}

Define the local Fourier extension characteristic $\operatorname*{Four}%
_{L^{q}\rightarrow L^{q}}^{S}\left(  R\right)  $ to be the best constant in
the inequality%
\[
\left\Vert \mathcal{E}_{S}f\right\Vert _{L^{q}\left(  B\left(  0,R\right)
\right)  }\leq\operatorname*{Four}\nolimits_{L^{q}\rightarrow L^{q}}%
^{S}\left(  R\right)  \ \left\Vert f\right\Vert _{L^{q}\left(  U\right)
},\ \text{for }f\in L^{q}\left(  U\right)  ,
\]
i.e. the operator norm of $\mathcal{E}_{S}$ from $L^{q}\left(  U\right)  $ to
$L^{q}\left(  B\left(  0,R\right)  \right)  $. Similarly, define
$\operatorname*{Bilin}_{\otimes_{2}L^{\infty}\rightarrow L^{\frac{q}{2}}}%
^{S}\left(  R\right)  $ to be the best constant in the bilinear inequality%
\[
\left\Vert \mathcal{E}_{S}f_{1}\ \mathcal{E}_{S}f_{2}\right\Vert _{L^{\frac
{q}{2}}\left(  B\left(  0,R\right)  \right)  }^{\frac{1}{2}}\leq
\operatorname*{BiFour}\nolimits_{\otimes_{2}L^{\infty}\rightarrow L^{\frac
{q}{2}}}^{S}\left(  R\right)  \ \left\Vert f_{1}\right\Vert _{L^{\infty
}\left(  U_{1}\right)  }^{\frac{1}{2}}\left\Vert f_{2}\right\Vert _{L^{\infty
}\left(  U_{2}\right)  }^{\frac{1}{2}},
\]
$\ $for $f_{\mu}\in L^{\infty}\left(  U_{\mu}\right)  $ where $U_{1}$ and
$U_{2}$ are $\nu$-separated. Then for $\nu>0$ sufficiently small, we have from
\cite{BoGu} that for every $\varepsilon>0$ there is a positive constant
$C_{\varepsilon}$ such that,%
\begin{equation}
\operatorname*{BiFour}\nolimits_{\otimes_{2}L^{\infty}\rightarrow L^{\frac
{q}{2}}}^{S}\left(  R\right)  \leq\operatorname*{Four}\nolimits_{L^{q}%
\rightarrow L^{q}}^{S}\left(  R\right)  \leq C_{\varepsilon}R^{\varepsilon
}\operatorname*{BiFour}\nolimits_{\otimes_{2}L^{\infty}\rightarrow L^{\frac
{q}{2}}}^{S}\left(  R\right)  .\label{well known}%
\end{equation}

Given a discrete subset $\Lambda\subset S$, define the \emph{discrete} Fourier
extension operator $\mathcal{D}_{S}^{\Lambda}$ as in \cite[Chapter 1]{Dem} by%
\[
\mathcal{D}_{S}^{\Lambda}a\left(  \xi\right)  =\sum_{\lambda\in\Lambda\subset
S}a_{\lambda}e^{i\xi\cdot\lambda},
\]
where $a=\left\{  a_{\lambda}\right\}  _{\lambda\in\Lambda}$ is a sequence of
complex numbers.

\begin{definition}
Let $1<q<\infty$. Define the local discrete characteristic
$\operatorname*{Disc}\nolimits_{\ell^{q}\rightarrow L^{q}}^{S}\left(
R\right)  $ to be the best constant in the testing inequalities,%
\begin{equation}
\left\Vert \mathcal{D}_{S}^{\Lambda}a\right\Vert _{L^{q}\left(  B_{\mathbb{R}%
^{2}}\left(  A,R\right)  \right)  }\leq\operatorname*{Disc}\nolimits_{\ell
^{q}\rightarrow L^{q}}^{S}\left(  R\right)  R^{\frac{1}{q^{\prime}}}\left\Vert
a\right\Vert _{\ell^{q}},\ \text{for }a\in\ell^{q}\left(  \Lambda\right)
,\label{discrete ineq}%
\end{equation}
when taken over all $\frac{1}{R}$-separated subsets $\Lambda$ of $S$ and all
disks $B_{\mathbb{R}^{2}}\left(  A,R\right)  $ in $\mathbb{R}^{2}$ (see e.g.
\cite[Theorem 1.29 on page 21]{Dem} for this definition).\newline In addition,
for $\Lambda_{s}=\Phi\left(  \left\{  \lambda_{I}\right\}  _{I\in
\mathcal{G}_{s}\left[  U\right]  }\right)  $, we define $\operatorname*{Disc}%
\nolimits_{\ell^{q}\rightarrow L^{q}}^{S;\mathcal{G}}\left(  R\right)  $ to be
the best constant in the testing inequality,%
\[
\left\Vert \mathcal{D}_{S}^{\Lambda_{s}}a\right\Vert _{L^{q}\left(
B_{\mathbb{R}^{2}}\left(  A,R\right)  \right)  }\leq\operatorname*{Disc}%
\nolimits_{\ell^{q}\rightarrow L^{q}}^{S;\mathcal{G}}\left(  R\right)
R^{\frac{1}{q^{\prime}}}\left\Vert a\right\Vert _{\ell^{q}},\ \text{for }%
a\in\ell^{q}\left(  \Lambda\right)  ,
\]
with $R=2^{s}$, taken over all $s\in\mathbb{N}$, all sets $\Phi\left(
\left\{  \lambda_{I}\right\}  _{I\in\mathcal{G}_{s}\left[  U\right]  }\right)
$ for $x\in\left[  0,2^{-s}\right)  ^{d-1}$, and all disks $B_{\mathbb{R}^{2}%
}\left(  A,2^{s}\right)  $ in $\mathbb{R}^{2}$.
\end{definition}

\begin{theorem}
\label{simple}Let $2<q<\infty$ and $R>1$. Then%
\[
\operatorname*{Four}\nolimits_{L^{q}\rightarrow L^{q}}^{S}\left(  R\right)
\approx\operatorname*{Disc}\nolimits_{\ell^{q}\rightarrow L^{q}}%
^{S;\mathcal{G}}\left(  R\right)  \approx\operatorname*{Disc}\nolimits_{\ell
^{q}\rightarrow L^{q}}^{S}\left(  R\right)  .
\]

\end{theorem}

The proof of Theorem \ref{simple} is an easy adaptation of the arguments in
Chapter 1 of C. Demeter's book \cite[Chapter 1]{Dem} using the following lemma
from \cite{TaVaVe}.

\begin{lemma}
[Tao, Vargas and Vega \cite{TaVaVe}]\label{fat surface}Let $\mathcal{N}%
_{R^{-1}}\left(  S\right)  $ denote the surface $S$ `fattened' to have
thickness $R^{-1}$. Then we have the inequality,%
\begin{equation}
\left\Vert \widehat{F}\right\Vert _{L^{q}\left(  B\left(  0,R\right)  \right)
}\lesssim R^{-\frac{1}{q^{\prime}}}\operatorname*{Four}\nolimits_{L^{q}%
\rightarrow L^{q}}^{S}\left\Vert F\right\Vert _{L^{q}\left(  \mathcal{N}%
_{R^{-1}}\left(  S\right)  \right)  }.\label{fat ineq}%
\end{equation}

\end{lemma}

In order to extend Theorem \ref{simple} to the bilinear setting, we first need
a direct proof of Lemma \ref{fat surface} that avoids duality. For example, if
we define the function $f$ by $f\left(  x\right)  \equiv\int_{-R^{-1}}%
^{R^{-1}}F\left(  \Phi\left(  x\right)  +t\right)  dt$, then
\begin{align}
&  \ \ \ \ \ \ \ \ \ \ \ \ \ \ \ \ \ \ \ \ \left\Vert \widehat{F}\right\Vert
_{L^{q}\left(  B\left(  0,R\right)  \right)  }^{q}=\int\left\vert \widehat
{F}\left(  \xi\right)  \right\vert ^{q}d\xi=\int\left\vert \int_{\mathcal{N}%
_{R^{-1}}\left(  S\right)  }e^{-2\pi i\xi\cdot z}F\left(  z\right)
dz\right\vert ^{q}d\xi\label{F^ to Ef}\\
&  =\int\left\vert \int_{-R^{-1}}^{R^{-1}}\left\{  \int e^{-2\pi i\xi
\cdot\left(  \Phi\left(  x\right)  +t\right)  }F\left(  \Phi\left(  x\right)
+t\right)  dx\right\}  dt\right\vert ^{q}d\xi=\int\left\vert \int_{-R^{-1}%
}^{R^{-1}}\left\{  \int e^{-2\pi i\xi\cdot\Phi\left(  x\right)  }F\left(
\Phi\left(  x\right)  +t\right)  dx\right\}  dt\right\vert ^{q}d\xi\nonumber\\
&  =\int\left\vert \left\{  \int e^{-2\pi i\xi\cdot\Phi\left(  x\right)
}\left(  \int_{-R^{-1}}^{R^{-1}}F\left(  \Phi\left(  x\right)  +t\right)
dt\right)  dx\right\}  \right\vert ^{q}d\xi=\int\left\vert \left\{  \int
e^{-2\pi i\xi\cdot\Phi\left(  x\right)  }f\left(  x\right)  dx\right\}
\right\vert ^{q}d\xi=\int\left\vert \mathcal{E}_{S}f\left(  \xi\right)
\right\vert ^{q}d\xi,\nonumber
\end{align}
and we can continue with%
\begin{align*}
&  \int\left\vert \mathcal{E}_{S}f\left(  \xi\right)  \right\vert ^{q}d\xi
\leq\left(  \operatorname*{Four}\nolimits_{L^{q}\rightarrow L^{q}}^{S}\right)
^{q}\int_{U}f\left(  x\right)  ^{q}dx=\left(  \operatorname*{Four}%
\nolimits_{L^{q}\rightarrow L^{q}}^{S}\right)  ^{q}\int_{U}\left\vert
\int_{-R^{-1}}^{R^{-1}}F\left(  \Phi\left(  x\right)  +t\right)  dt\right\vert
^{q}dx\\
&  \leq\left(  \operatorname*{Four}\nolimits_{L^{q}\rightarrow L^{q}}%
^{S}\right)  ^{q}\int_{U}\left\vert \left(  R^{-1}\right)  ^{\frac
{1}{q^{\prime}}}\left(  \int_{-R^{-1}}^{R^{-1}}\left\vert F\left(  \Phi\left(
x\right)  +t\right)  \right\vert ^{q}dt\right)  ^{\frac{1}{q}}\right\vert
^{q}dx\\
&  =\left(  \operatorname*{Four}\nolimits_{L^{q}\rightarrow L^{q}}^{S}\right)
^{q}\left(  R^{-\frac{1}{q^{\prime}}}\right)  ^{q}\int_{U}\int_{-R^{-1}%
}^{R^{-1}}\left\vert F\left(  \Phi\left(  x\right)  +t\right)  \right\vert
^{q}dtdx\approx\left(  \operatorname*{Four}\nolimits_{L^{q}\rightarrow L^{q}%
}^{S}\right)  ^{q}\left(  R^{-\frac{1}{q^{\prime}}}\right)  ^{q}%
\int_{\mathcal{N}_{R^{-1}}\left(  S\right)  }\left\vert F\left(  z\right)
\right\vert ^{q}dz,
\end{align*}
which proves (\ref{fat ineq}).

\begin{definition}
Let $1<q<\infty$. Define the local discrete bilinear characteristic
$\operatorname*{BiDisc}\nolimits_{\ell^{q}\rightarrow L^{\frac{q}{2}}}%
^{S}\left(  R\right)  $ to be the best constant in the testing inequalities,%
\[
\left\Vert \prod_{\mu=1}^{2}\mathcal{D}_{S}^{\Lambda_{\mu}}a\right\Vert
_{L^{q}\left(  B_{\mathbb{R}^{2}}\left(  A,R\right)  \right)  }^{\frac{1}{2}%
}\leq\operatorname*{BiDisc}\nolimits_{\otimes_{2}\ell^{q}\rightarrow
L^{\frac{q}{2}}}^{S}\left(  R\right)  R^{\frac{1}{q^{\prime}}}\prod_{\mu
=1}^{2}\left\Vert a_{\mu}\right\Vert _{\ell^{q}}^{\frac{1}{2}},\ \text{for
}a_{\mu}\in\ell^{q}\left(  \Lambda_{\mu}\right)  ,
\]
when taken over all $\frac{1}{R}$-separated subsets $\Lambda_{\mu}$ of $S$ and
all disks $B_{\mathbb{R}^{2}}\left(  A,R\right)  $ in $\mathbb{R}^{2}$ and
where the discrete sets $\Lambda_{1}$ and $\Lambda_{2}$ are $\nu
$-separated.\newline In addition, for $\Lambda_{s}=\Phi\left(  \left\{
\lambda_{I}\right\}  _{I\in\mathcal{G}_{s}\left[  U\right]  }\right)  $, we
define $\operatorname*{BiDisc}\nolimits_{\otimes_{2}\ell^{q}\rightarrow
L^{\frac{q}{2}}}^{S;\mathcal{G}}\left(  R\right)  $ to be the best constant in
the testing inequality,%
\[
\left\Vert \prod_{\mu=1}^{2}\mathcal{D}_{S}^{\Lambda_{\mu};\mathcal{G}%
}a\right\Vert _{L^{\frac{q}{2}}\left(  B_{\mathbb{R}^{2}}\left(  A,R\right)
\right)  }^{\frac{1}{2}}\leq\operatorname*{BiDisc}\nolimits_{\otimes_{2}%
\ell^{q}\rightarrow L^{\frac{q}{2}}}^{S;\mathcal{G}}\left(  R\right)
R^{\frac{1}{q^{\prime}}}\prod_{\mu=1}^{2}\left\Vert a_{\mu}\right\Vert
_{\ell^{q}}^{\frac{1}{2}},\ \text{for }a_{\mu}\in\ell^{q}\left(  \Lambda_{\mu
}\right)  ,
\]
with $R=2^{s}$, taken over all $s\in\mathbb{N}$, all sets $\Lambda_{\mu}%
=\Phi\left(  \left\{  \lambda_{I}\right\}  _{I\in\mathcal{G}_{s}\left[
U_{\mu}\right]  }\right)  $, and all disks $B_{\mathbb{R}^{2}}\left(
A,2^{s}\right)  $ in $\mathbb{R}^{2}$, and where the discrete sets
$\Lambda_{1}$ and $\Lambda_{2}$\ are $\nu$-separated.
\end{definition}

Now recall%
\[
\left\Vert \prod_{\mu=1}^{2}\mathcal{E}_{S}f_{\mu}\right\Vert _{L^{\frac{q}%
{2}}}^{\frac{1}{2}}\leq\operatorname*{BiFour}\nolimits_{\otimes_{2}%
L^{q}\rightarrow L^{\frac{q}{2}}}^{S}\prod_{\mu=1}^{2}\left\Vert f_{\mu
}\right\Vert _{L^{q}\left(  U\right)  }^{\frac{1}{2}}\ ,
\]
and apply the above direct argument with $f_{\mu}\left(  x\right)  \equiv
\int_{-R^{-1}}^{R^{-1}}F_{\mu}\left(  \Phi\left(  x\right)  +t\right)  dt$, in
the bilinear setting\ to obtain%
\begin{align*}
&  \left\Vert \prod_{\mu=1}^{2}\widehat{F_{\mu}}\right\Vert _{L^{\frac{q}{2}%
}\left(  B\left(  0,R\right)  \right)  }^{\frac{q}{2}}=\left\Vert \prod
_{\mu=1}^{2}\mathcal{E}_{S}f_{\mu}\right\Vert _{L^{\frac{q}{2}}}^{\frac{q}{2}%
}\leq\left(  \operatorname*{BiFour}\nolimits_{\otimes_{2}L^{q}\rightarrow
L^{\frac{q}{2}}}^{S}\right)  ^{q}\prod_{\mu=1}^{2}\left\Vert f_{\mu
}\right\Vert _{L^{q}\left(  U\right)  }^{\frac{1}{2}}\\
&  =\left(  \operatorname*{BiFour}\nolimits_{\otimes_{2}L^{q}\rightarrow
L^{\frac{q}{2}}}^{S}\right)  ^{q}\prod_{\mu=1}^{2}\left(  \int_{U}\left\vert
\int_{-R^{-1}}^{R^{-1}}F_{\mu}\left(  \Phi\left(  x\right)  +t\right)
dt\right\vert ^{q}dx\right)  ^{\frac{1}{2}}\\
&  \leq\left(  \operatorname*{BiFour}\nolimits_{\otimes_{2}L^{q}\rightarrow
L^{\frac{q}{2}}}^{S}\right)  ^{q}\prod_{\mu=1}^{2}\left(  \int_{U}\left\vert
\left(  R^{-1}\right)  ^{\frac{1}{q^{\prime}}}\left(  \int_{-R^{-1}}^{R^{-1}%
}\left\vert F_{\mu}\left(  \Phi\left(  x\right)  +t\right)  \right\vert
^{q}dt\right)  ^{\frac{1}{q}}\right\vert ^{q}dx\right)  ^{\frac{1}{2}}\\
&  =\left(  \operatorname*{BiFour}\nolimits_{\otimes_{2}L^{q}\rightarrow
L^{\frac{q}{2}}}^{S}\right)  ^{q}\left(  R^{-\frac{1}{q^{\prime}}}\right)
^{q}\prod_{\mu=1}^{2}\left(  \int_{U}\int_{-R^{-1}}^{R^{-1}}\left\vert F_{\mu
}\left(  \Phi\left(  x\right)  +t\right)  \right\vert ^{q}dtdx\right)
^{\frac{1}{2}}\\
&  \approx\left(  \operatorname*{BiFour}\nolimits_{\otimes_{2}L^{q}\rightarrow
L^{\frac{q}{2}}}^{S}\right)  ^{q}\left(  R^{-\frac{1}{q^{\prime}}}\right)
^{q}\prod_{\mu=1}^{2}\left(  \int_{\mathcal{N}_{R^{-1}}\left(  S\right)
}\left\vert F_{\mu}\left(  z\right)  \right\vert ^{q}dz\right)  ^{\frac{1}{2}%
}=\left(  \operatorname*{BiFour}\nolimits_{\otimes_{2}L^{q}\rightarrow
L^{\frac{q}{2}}}^{S}\right)  ^{q}\left(  R^{-\frac{1}{q^{\prime}}}\right)
^{q}\prod_{\mu=1}^{2}\left\Vert F_{\mu}\right\Vert _{L^{q}\left(
\mathcal{N}_{R^{-1}}\left(  S\right)  \right)  }^{\frac{q}{2}},
\end{align*}
which gives the fattened form of the bilinear Fourier extension inequality,%
\[
\left\Vert \prod_{\mu=1}^{2}\widehat{F_{\mu}}\right\Vert _{L^{\frac{q}{2}%
}\left(  B\left(  0,R\right)  \right)  }^{\frac{1}{2}}\lesssim
\operatorname*{BiFour}\nolimits_{\otimes_{2}L^{q}\rightarrow L^{\frac{q}{2}}%
}^{S}R^{-\frac{1}{q^{\prime}}}\prod_{\mu=1}^{2}\left\Vert F_{\mu}\right\Vert
_{L^{q}\left(  \mathcal{N}_{R^{-1}}\left(  S\right)  \right)  }^{\frac{1}{2}}.
\]

\subsubsection{Discrete bilinear characterization}

Now we can state and prove the bilinear analogue of Theorem \ref{simple},
which will be used in conjunction with (\ref{well known}) above.

\begin{theorem}
\label{bilinear simple}Let $2<q<\infty$ and $R>1$. Then%
\[
\operatorname*{BiFour}\nolimits_{\otimes_{2}L^{q}\rightarrow L^{\frac{q}{2}}%
}^{S}\left(  R\right)  \approx\operatorname*{BiDisc}\nolimits_{\otimes_{2}%
\ell^{q}\rightarrow L^{\frac{q}{2}}}^{S;\mathcal{G}}\left(  R\right)
\approx\operatorname*{BiDisc}\nolimits_{\otimes_{2}\ell^{q}\rightarrow
L^{\frac{q}{2}}}^{S}\left(  R\right)  .
\]

\end{theorem}

\begin{proof}
We write%
\begin{align}
&  \left\Vert \prod_{\mu=1}^{2}\sum_{\lambda_{\mu}\in\Lambda_{\mu}\subset
S}a_{\lambda_{\mu}}e^{i\xi\cdot\lambda_{\mu}}\right\Vert _{L^{\frac{q}{2}%
}\left(  B_{\mathbb{R}^{2}}\left(  \mathbf{c},R\right)  \right)  }^{\frac
{1}{2}}\leq\left\Vert \widehat{\varphi_{R^{-1}}}\left(  \xi\right)  ^{2}%
\prod_{\mu=1}^{2}\sum_{\lambda_{\mu}\in\Lambda_{\mu}\subset S}a_{\lambda_{\mu
}}e^{i\left(  \xi-\mathbf{c}\right)  \cdot\lambda_{\mu}}\right\Vert
_{L^{\frac{q}{2}}\left(  B_{\mathbb{R}^{2}}\left(  0,R\right)  \right)
}^{\frac{1}{2}}\label{a to F}\\
&  =\left\Vert \prod_{\mu=1}^{2}\sum_{\lambda_{\mu}\in\Lambda_{\mu}\subset
S}b_{\lambda_{\mu}}\widehat{\varphi_{R^{-1}}\left(  \cdot+\lambda_{\mu
}\right)  }\right\Vert _{L^{\frac{q}{2}}\left(  B_{\mathbb{R}^{2}}\left(
0,R\right)  \right)  }^{\frac{1}{2}}=\left\Vert \prod_{\mu=1}^{2}%
\widehat{F_{\mu}}\right\Vert _{L^{q}\left(  B_{\mathbb{R}^{2}}\left(
0,R\right)  \right)  }^{\frac{1}{2}},\nonumber
\end{align}
where $b_{\lambda_{\mu}}=e^{-i\mathbf{c}\cdot\lambda_{\mu}}a_{\lambda_{\mu}}$
and $F_{\mu}\left(  z\right)  =\sum_{\lambda_{\mu}\in\Lambda_{\mu}\subset
S}b_{\lambda_{\mu}}\varphi_{R^{-1}}\left(  z+\lambda_{\mu}\right)  $. Next we
continue with%
\begin{align*}
&  \left\Vert \prod_{\mu=1}^{2}\widehat{F_{\mu}}\right\Vert _{L^{q}\left(
B_{\mathbb{R}^{2}}\left(  0,R\right)  \right)  }^{\frac{1}{2}}\lesssim
\operatorname*{BiFour}\nolimits_{\otimes_{2}L^{q}\rightarrow L^{\frac{q}{2}}%
}^{S}\left(  R\right)  R^{-\frac{1}{q^{\prime}}}\prod_{\mu=1}^{2}\left\Vert
\sum_{\lambda_{\mu}\in\Lambda_{\mu}\subset S}b_{\lambda_{\mu}}\varphi_{R^{-1}%
}\left(  z+\lambda_{\mu}\right)  \right\Vert _{L^{q}\left(  \mathcal{N}%
_{R^{-1}}\left(  S\right)  \right)  }^{\frac{1}{2}}\\
&  \approx\operatorname*{BiFour}\nolimits_{\otimes_{2}L^{q}\rightarrow
L^{\frac{q}{2}}}^{S}R^{-\frac{1}{q^{\prime}}}\prod_{\mu=1}^{2}\left\vert
b\right\vert _{\ell^{q}}^{\frac{1}{2}}\left\Vert \varphi_{R^{-1}}\right\Vert
_{L^{q}\left(  \mathbb{R}^{2}\right)  }\approx R^{\frac{1}{q^{\prime}}%
}\operatorname*{BiFour}\nolimits_{\otimes_{2}L^{q}\rightarrow L^{\frac{q}{2}}%
}^{S}\left(  R\right)  \left\vert a\right\vert _{\ell^{q}}\ ,
\end{align*}
since $\left\Vert \varphi_{R^{-1}}\right\Vert _{L^{q}\left(  \mathbb{R}%
^{d}\right)  }\approx R^{2}R^{-\frac{2}{q}}$, which shows that
$\operatorname*{BiDisc}\nolimits_{\otimes_{2}\ell^{q}\rightarrow L^{\frac
{q}{2}}}^{S}\left(  R\right)  \lesssim\operatorname*{BiFour}\nolimits_{\otimes
_{2}L^{q}\rightarrow L^{\frac{q}{2}}}^{S}\left(  R\right)  $.

Finally, we see that H\"{o}lder's inequality, followed by Fubini's theorem,
gives%
\begin{align*}
&  \left\Vert \prod_{\mu=1}^{2}\mathcal{E}f_{\mu}\right\Vert _{L^{\frac{q}{2}%
}\left(  B\left(  0,R\right)  \right)  }=\left\Vert \prod_{\mu=1}^{2}%
\int_{\left[  0,2^{-s}\right)  }\sum_{I_{\mu}\in\mathcal{G}_{s}\left[  U_{\mu
}\right]  }f\left(  \lambda_{I_{\mu}}+x\right)  e^{-2\pi i\Phi\left(
\lambda_{I_{\mu}}+x\right)  \cdot\xi}dx\right\Vert _{L^{\frac{q}{2}}\left(
B\left(  0,R\right)  \right)  }\\
&  \leq\left\Vert \prod_{\mu=1}^{2}2^{-\frac{1}{q^{\prime}}s}\left\Vert
\sum_{I_{\mu}\in\mathcal{G}_{s}\left[  U_{\mu}\right]  }f\left(
\lambda_{I_{\mu}}+x\right)  e^{-2\pi i\Phi\left(  \lambda_{I_{\mu}}+x\right)
\cdot\xi}\right\Vert _{L^{q}\left(  \left[  0,2^{-s}\right)  \right)
}\right\Vert _{L^{\frac{q}{2}}\left(  B\left(  0,R\right)  \right)  }\\
&  =2^{-\frac{2}{q^{\prime}}s}\left\Vert \left\Vert \prod_{\mu=1}^{2}%
\sum_{I_{\mu}\in\mathcal{G}_{s}\left[  U_{\mu}\right]  }f\left(
\lambda_{I_{\mu}}+x\right)  e^{-2\pi i\Phi\left(  \lambda_{I_{\mu}}+x\right)
\cdot\xi}\right\Vert _{L^{q}\left(  B\left(  0,R\right)  \right)  }\right\Vert
_{L^{\frac{q}{2}}\left(  \left[  0,2^{-s}\right)  \right)  }\\
&  \leq\operatorname*{BiDisc}\nolimits_{\otimes_{2}\ell^{q}\rightarrow
L^{\frac{q}{2}}}^{S}\left(  R\right)  ^{2}2^{-\frac{2}{q^{\prime}}s}\left\Vert
\prod_{\mu=1}^{2}2^{\frac{1}{q^{\prime}}s}\left\Vert f\left(  \lambda_{I_{\mu
}}+x\right)  \right\Vert _{\ell^{q}\left(  \mathcal{G}_{s}\left[  U_{\mu
}\right]  \right)  }\right\Vert _{L^{\frac{q}{2}}\left(  \left[
0,2^{-s}\right)  \right)  },
\end{align*}
which equals%
\begin{align*}
&  \operatorname*{BiDisc}\nolimits_{\otimes_{2}\ell^{q}\rightarrow L^{\frac
{q}{2}}}^{S}\left(  R\right)  ^{2}\left(  \int_{\left[  0,2^{-s}\right)
}\left(  \sum_{I_{1}\in\mathcal{G}_{s}\left[  U_{1}\right]  }\left\vert
f\left(  \lambda_{I_{1}}+x\right)  \right\vert ^{q}\right)  ^{\frac{1}{2}%
}\left(  \sum_{I_{2}\in\mathcal{G}_{s}\left[  U_{2}\right]  }\left\vert
f\left(  \lambda_{I_{2}}+x\right)  \right\vert ^{q}\right)  ^{\frac{1}{2}%
}dx\right)  ^{\frac{2}{q}}\\
&  \leq\operatorname*{BiDisc}\nolimits_{\otimes_{2}\ell^{q}\rightarrow
L^{\frac{q}{2}}}^{S}\left(  R\right)  ^{2}\left(  \int_{\left[  0,2^{-s}%
\right)  }\sum_{I_{1}\in\mathcal{G}_{s}\left[  U_{1}\right]  }\left\vert
f\left(  \lambda_{I_{1}}+x\right)  \right\vert ^{q}dx\right)  ^{\frac{1}{q}%
}\left(  \int_{\left[  0,2^{-s}\right)  }\sum_{I_{2}\in\mathcal{G}_{s}\left[
U_{2}\right]  }\left\vert f\left(  \lambda_{I_{2}}+x\right)  \right\vert
^{q}dx\right)  ^{\frac{1}{q}}\\
&  =\operatorname*{BiDisc}\nolimits_{\otimes_{2}\ell^{q}\rightarrow
L^{\frac{q}{2}}}^{S}\left(  R\right)  ^{2}\left\Vert f_{1}\right\Vert _{L^{q}%
}\left\Vert f_{2}\right\Vert _{L^{q}},
\end{align*}
which proves%
\[
\operatorname*{BiFour}\nolimits_{\otimes_{2}L^{q}\rightarrow L^{\frac{q}{2}}%
}^{S}\left(  R\right)  \lesssim\operatorname*{BiDisc}\nolimits_{\otimes
_{2}\ell^{q}\rightarrow L^{\frac{q}{2}}}^{S}\left(  R\right)  \ ,
\]
and completes the proof of Theorem \ref{bilinear simple}.
\end{proof}

\begin{remark}
The Cauchy-Schwarz inequality implies the elementary result
\[
\operatorname*{BiDisc}\nolimits_{\otimes_{2}\ell^{q}\rightarrow L^{\frac{q}%
{2}}}^{S}\left(  R\right)  \lesssim\operatorname*{Disc}\nolimits_{\ell
^{q}\rightarrow L^{q}}^{S}\left(  R\right)  \lesssim C_{\varepsilon
}R^{\varepsilon}\operatorname*{BiDisc}\nolimits_{\otimes_{2}\ell
^{q}\rightarrow L^{\frac{q}{2}}}^{S}\left(  R\right)  ,
\]
which yields an easy proof that%
\[
\operatorname*{BiDisc}\nolimits_{\otimes_{2}\ell^{q}\rightarrow L^{\frac{q}%
{2}}}^{S}\left(  R\right)  \approx_{\varepsilon}\operatorname*{BiFour}%
\nolimits_{\otimes_{2}L^{q}\rightarrow L^{\frac{q}{2}}}^{S}\left(  R\right)  .
\]
Indeed,%
\begin{align*}
&  \left\Vert \prod_{\mu=1}^{2}\mathcal{D}_{S}^{\Lambda_{\mu}}a_{\mu
}\right\Vert _{L^{\frac{q}{2}}\left(  B_{\mathbb{R}^{2}}\left(  A,R\right)
\right)  }^{\frac{q}{2}}=\int_{B_{\mathbb{R}^{2}}\left(  A,R\right)
}\left\vert \mathcal{D}_{S}^{\Lambda_{1}}a_{1}\left(  \xi\right)  \right\vert
^{\frac{q}{2}}\left\vert \mathcal{D}_{S}^{\Lambda_{2}}a_{2}\left(  \xi\right)
\right\vert ^{\frac{q}{2}}d\xi\\
&  \leq\left(  \int_{B_{\mathbb{R}^{2}}\left(  A,R\right)  }\left\vert
\mathcal{D}_{S}^{\Lambda_{1}}a_{1}\left(  \xi\right)  \right\vert ^{q}%
d\xi\right)  ^{\frac{1}{2}}\left(  \int_{B_{\mathbb{R}^{2}}\left(  A,R\right)
}\left\vert \mathcal{D}_{S}^{\Lambda_{2}}a_{2}\left(  \xi\right)  \right\vert
^{q}d\xi\right)  ^{\frac{1}{2}}\\
&  \leq\left(  \operatorname*{Disc}\nolimits_{\ell^{q}\rightarrow L^{q}}%
^{S}\left(  R\right)  ^{q}R^{\frac{q}{q^{\prime}}}\left\Vert a_{1}\right\Vert
_{\ell^{q}}^{q}\right)  ^{\frac{1}{2}}\left(  \operatorname*{Disc}%
\nolimits_{\ell^{q}\rightarrow L^{q}}^{S}\left(  R\right)  ^{q}R^{\frac
{q}{q^{\prime}}}\prod_{\mu=1}^{2}\left\Vert a_{2}\right\Vert _{\ell^{q}}%
^{q}\right)  ^{\frac{1}{2}}\\
&  =\operatorname*{Disc}\nolimits_{\ell^{q}\rightarrow L^{q}}^{S}\left(
R\right)  R^{\frac{1}{q^{\prime}}}\left\Vert a_{1}\right\Vert _{\ell^{q}%
}^{\frac{1}{2}}\left\Vert a_{2}\right\Vert _{\ell^{q}}^{\frac{1}{2}},
\end{align*}
which shows that $\operatorname*{BiDisc}\nolimits_{\otimes_{2}\ell
^{q}\rightarrow L^{\frac{q}{2}}}^{S}\left(  R\right)  \leq\operatorname*{Disc}%
\nolimits_{\ell^{q}\rightarrow L^{q}}^{S}\left(  R\right)  $. Conversely we
have
\[
\operatorname*{Disc}\nolimits_{\ell^{q}\rightarrow L^{q}}^{S}\left(  R\right)
\leq\operatorname*{Four}\nolimits_{L^{q}\rightarrow L^{q}}^{S}\left(
R\right)  \leq C_{\varepsilon}R^{\varepsilon}\operatorname*{BiFour}%
\nolimits_{\otimes_{2}L^{q}\rightarrow L^{\frac{q}{2}}}^{S}\left(  R\right)
\approx C_{\varepsilon}R^{\varepsilon}\operatorname*{BiDisc}\nolimits_{\otimes
_{2}\ell^{q}\rightarrow L^{\frac{q}{2}}}^{S}\left(  R\right)  .
\]

\end{remark}

\subsection{Step 4: bilinear wave packet estimates}

We now estimate the $L^{\frac{q}{2}}\left(  Q\left(  0,2^{\frac{s}{1-\delta}%
}\right)  \right)  $ norm of the product function,%
\begin{align}
&  \prod_{\mu=1}^{2}\mathcal{E}f_{\mu}\left(  \xi\right)  =\sum_{\alpha}%
\prod_{\mu=1}^{2}R^{-\frac{1}{2}}\sum_{\substack{\left(  K_{\mu},L_{\mu
}\right)  \in\mathcal{G}_{-\frac{1}{2}s}\left(  \left[  -2^{s},2^{s}\right]
\right)  \times\mathcal{G}_{\frac{1}{2}s}\left[  U_{\mu}\right]  \\\left(
K_{\mu},L_{\mu}\right)  \sim\alpha}}\widetilde{f_{\mu,L}}\left(
\lambda_{K_{\mu}}\right)  e^{-2\pi i\left(  \xi_{1}\ell_{\mu}+\xi_{2}\ell
_{\mu}^{2}\right)  }\label{must now}\\
&  =R^{-1}\sum_{\alpha}\sum_{\substack{\left(  K_{1},L_{1}\right)
\in\mathcal{G}_{-\frac{1}{2}s}\left(  \left[  -2^{s},2^{s}\right]  \right)
\times\mathcal{G}_{\frac{1}{2}s}\left(  \left[  -1,1\right]  \right)
\\\left(  K_{1},L_{1}\right)  \sim\alpha}}\nonumber\\
&  \times\sum_{\substack{\left(  K_{2}^{\prime},L_{2}^{\prime}\right)
\in\mathcal{G}_{-\frac{1}{2}s}\left(  \left[  -2^{s},2^{s}\right]  \right)
\times\mathcal{G}_{\frac{1}{2}s}\left[  U_{\mu}\right]  \\\left(  K_{2}%
,L_{2}\right)  \sim\alpha}}\widetilde{f_{1,L}}\left(  \lambda_{K_{1}}\right)
\widetilde{f_{2,L}}\left(  \lambda_{K_{2}}\right)  e^{-2\pi i\xi_{1}\left(
\ell_{1}+\ell_{2}\right)  +\xi_{2}\left(  \ell_{1}^{2}+\ell_{2}^{2}\right)
},\nonumber
\end{align}
where%
\[
\widetilde{f_{\mu,L_{\mu}}}\left(  \lambda_{K_{\mu}}\right)  =R^{\frac{1}{4}%
}\widehat{f_{\mu,L_{\mu}}}\left(  \lambda_{K_{\mu}}\right)  ,
\]
and where for a fixed square $\alpha$, we write $\left(  K,L\right)
\sim\alpha$ to mean $T_{K,L}\cap\alpha\neq\emptyset$.

However, for convenience we use here the smaller square $Q\left(
0,2^{s}\right)  $, as the larger square $Q\left(  0,2^{\frac{s}{1-\delta}%
}\right)  $ only introduces an acceptable small growth factor of the form
$2^{c\delta}$. Thus we decompose the upper rectangle $Q_{+}\left(
0,2^{s}\right)  $ into roughly $R$ squares $\alpha$, each of side length
$R^{\frac{1}{2}}$. Then for each $L\in\mathcal{G}_{\frac{1}{2}s}\left[
U\right]  $ there is essentially a unique $K\in\mathcal{G}_{-\frac{1}{2}%
s}\left(  \left[  -2^{s},2^{s}\right]  \right)  $ such that $\alpha\cap
T_{K,L}\neq\emptyset$.

We begin with the wave packet formula for the product $\prod_{\mu=1}%
^{2}\mathbb{E}_{\mathbb{T}_{s}}\mathcal{E}f_{\mu}$ from (\ref{must now}),%
\begin{align*}
&  \prod_{\mu=1}^{2}\mathcal{E}f_{\mu}\left(  \xi\right)  =R^{-\frac{1}{2}%
}\sum_{\alpha}\sum_{\substack{\left(  K_{1},L_{1}\right)  \in\mathcal{G}%
_{-\frac{1}{2}s}\left[  \left[  -2^{s},2^{s}\right)  \right]  \times
\mathcal{G}_{\frac{1}{2}s}\left[  U_{1}\right]  \\\left(  K_{1},L_{1}\right)
\sim Q}}\\
&  \ \ \ \ \ \ \ \ \ \ \times\sum_{\substack{\left(  K_{2}^{\prime}%
,L_{2}^{\prime}\right)  \in\mathcal{G}_{-\frac{1}{2}s}\left[  \left[
-2^{s},2^{s}\right)  \right]  \times\mathcal{G}_{\frac{1}{2}s}\left[
U_{2}\right]  \\\left(  K_{2},L_{2}\right)  \sim\alpha}}\widehat{f_{1,L}%
}\left(  \lambda_{K_{1}}\right)  \widehat{f_{2,L}}\left(  \lambda_{K_{2}%
}\right)  e^{-2\pi i\xi_{1}\left(  \ell_{1}+\ell_{2}\right)  +\xi_{2}\left(
\ell_{1}^{2}+\ell_{2}^{2}\right)  }\mathbf{1}_{\alpha}\left(  \xi\right)  ,
\end{align*}
which is consistent with the wave packet formula for $\mathcal{E}f_{\mu}$ in
(\ref{wp}),%
\[
\mathcal{E}f_{\mu}\left(  \xi\right)  =R^{-\frac{1}{4}}\sum_{\alpha}%
\sum_{\left(  K_{\mu},L_{\mu}\right)  \sim\alpha}e^{-2\pi i\xi\cdot\Phi\left(
\lambda_{L_{\mu}}\right)  }\widehat{f_{L_{\mu}}}\left(  \lambda_{K_{\mu}%
}\right)  \mathbf{1}_{T_{K_{\mu},L_{\mu}}}\left(  \xi\right)  .
\]
Then we compute the power norm%
\begin{align*}
&  \left\Vert \prod_{\mu=1}^{2}\mathcal{E}f_{\mu}\right\Vert _{L^{\frac{q}{2}%
}\left(  B\left(  0,R\right)  \right)  }^{\frac{q}{2}}=R^{-1}\sum_{\alpha}%
\int_{\alpha}\prod_{\mu=1}^{2}\left\vert \sum_{\left(  K_{\mu},L_{\mu}\right)
\sim\alpha}e^{-2\pi i\xi\cdot\Phi\left(  \lambda_{L_{\mu}}\right)  }%
\widehat{f_{L_{\mu}}}\left(  \lambda_{K_{\mu}}\right)  \right\vert ^{\frac
{q}{2}}d\xi\\
&  =R^{-1}\sum_{\alpha}\int_{\alpha}\prod_{\mu=1}^{2}\left\vert \sum_{L_{\mu
}\in\mathcal{G}_{\frac{1}{2}s}\left[  U_{\mu}\right]  }\widehat{f_{L_{\mu}}%
}\left(  K_{L_{\mu},\alpha}\right)  e^{-2\pi i\xi\cdot\Phi\left(
\lambda_{L_{\mu}}\right)  }\right\vert ^{\frac{q}{2}}d\xi=R^{-1}\sum_{\alpha
}\int_{\alpha}\prod_{\mu=1}^{2}\left\vert \mathcal{D}_{S}^{\Lambda_{\mu}%
}a_{\mu}^{\alpha}\left(  \xi\right)  \right\vert ^{\frac{q}{2}}d\xi
\end{align*}
where we recall,%
\begin{align*}
\mathcal{D}_{S}^{\Lambda_{\mu}}a_{\mu}^{\alpha}\left(  \xi\right)   &
=\sum_{\lambda_{\mu}\in\Lambda_{\mu}\subset S}a_{\lambda_{\mu}}^{\alpha
}e^{i\xi\cdot\lambda_{\mu}}=\sum_{L_{\mu}\in\mathcal{G}_{\frac{1}{2}s}\left[
U_{\mu}\right]  }\widehat{f_{L_{\mu}}}\left(  K_{L_{\mu},\alpha}\right)
e^{i\xi\cdot\Phi\left(  \lambda_{L_{\mu}}\right)  },\\
\text{where }a_{\mu}^{\alpha}  &  \equiv\left\{  \widehat{f_{L_{\mu}}}\left(
K_{L_{\mu},\alpha}\right)  \right\}  _{L_{\mu}\in\mathcal{G}_{\frac{1}{2}%
s}\left[  U_{\mu}\right]  }\text{ and }\Lambda_{\mu}\equiv\left\{  \Phi\left(
\lambda_{L_{\mu}}\right)  \right\}  _{L_{\mu}\in\mathcal{G}_{\frac{1}{2}%
s}\left[  U_{\mu}\right]  }.
\end{align*}

Next we rewrite the power norm of $\prod_{\mu=1}^{2}\mathcal{D}_{S}%
^{\Lambda_{\mu}}a_{\mu}^{\alpha}$\ over the square $\alpha$ into a power norm
of an associated Fourier extension over $\alpha$. For this we combine the two
inequalities (\ref{a to F}) and (\ref{F^ to Ef}), i.e.%
\begin{align*}
&  \left\Vert \prod_{\mu=1}^{2}\sum_{\lambda_{\mu}\in\Lambda_{\mu}\subset
S}a_{\lambda_{\mu}}e^{i\xi\cdot\lambda_{\mu}}\right\Vert _{L^{\frac{q}{2}%
}\left(  B_{\mathbb{R}^{2}}\left(  \mathbf{c},R\right)  \right)  }^{\frac
{1}{2}}\leq\left\Vert \widehat{\varphi_{R^{-1}}}\left(  \xi\right)  ^{2}%
\prod_{\mu=1}^{2}\sum_{\lambda_{\mu}\in\Lambda_{\mu}\subset S}a_{\lambda_{\mu
}}^{\alpha}e^{i\left(  \xi-\mathbf{c}_{\alpha}\right)  \cdot\lambda_{\mu}%
}\right\Vert _{L^{\frac{q}{2}}\left(  B_{\mathbb{R}^{2}}\left(  0,R\right)
\right)  }^{\frac{1}{2}}\\
&  =\left\Vert \prod_{\mu=1}^{2}\sum_{\lambda_{\mu}\in\Lambda_{\mu}\subset
S}b_{\lambda_{\mu}}^{\alpha}\widehat{\varphi_{R^{-1}}\left(  \cdot+\left(
0,\lambda_{\mu}\right)  \right)  }\right\Vert _{L^{\frac{q}{2}}\left(
B_{\mathbb{R}^{2}}\left(  0,R\right)  \right)  }^{\frac{1}{2}}=\left\Vert
\prod_{\mu=1}^{2}\widehat{F_{\mu}^{\alpha}}\right\Vert _{L^{q}\left(
B_{\mathbb{R}^{2}}\left(  0,R\right)  \right)  }^{\frac{1}{2}},
\end{align*}
where
\[
b_{\lambda_{\mu}}^{\alpha}=e^{-i\mathbf{c}_{\alpha}\cdot\Phi\left(
\lambda_{\mu}\right)  }a_{\lambda_{\mu}}\text{ and }F_{\mu}^{\alpha\mathbf{c}%
}\left(  z\right)  =\sum_{\lambda_{\mu}\in\Lambda_{\mu}\subset S}%
b_{\lambda_{\mu}}^{\alpha}\varphi_{R^{-1}}\left(  z+\left(  0,\lambda_{\mu
}\right)  \right)  ,
\]
and%
\begin{align*}
&  \left\Vert \widehat{F}\right\Vert _{L^{q}\left(  B\left(  0,R\right)
\right)  }^{q}=\int\left\vert \widehat{F}\left(  \xi\right)  \right\vert
^{q}d\xi=\int\left\vert \int_{\mathcal{N}_{R^{-1}}\left(  S\right)  }e^{-2\pi
i\xi\cdot z}F\left(  z\right)  dz\right\vert ^{q}d\xi\\
&  =\int\left\vert \int_{-R^{-1}}^{R^{-1}}\left\{  \int e^{-2\pi i\xi
\cdot\left(  \Phi\left(  x\right)  +\left(  0,t\right)  \right)  }F\left(
\Phi\left(  x\right)  +\left(  0,t\right)  \right)  dx\right\}  dt\right\vert
^{q}d\xi\\
&  =\int\left\vert \int_{-R^{-1}}^{R^{-1}}\left\{  \int e^{-2\pi i\xi\cdot
\Phi\left(  x\right)  }F\left(  \Phi\left(  x\right)  +\left(  0,t\right)
\right)  dx\right\}  dt\right\vert ^{q}d\xi\\
&  =\int\left\vert \int e^{-2\pi i\xi\cdot\Phi\left(  x\right)  }\left\{
\int_{-R^{-1}}^{R^{-1}}F\left(  \Phi\left(  x\right)  +\left(  0,t\right)
\right)  dt\right\}  dx\right\vert ^{q}d\xi=\int\left\vert \int e^{-2\pi
i\xi\cdot\Phi\left(  x\right)  }f\left(  x\right)  dx\right\vert ^{q}d\xi
=\int\left\vert \mathcal{E}_{S}f\left(  \xi\right)  \right\vert ^{q}d\xi,
\end{align*}
where
\[
f\left(  x\right)  \equiv\int_{-R^{-1}}^{R^{-1}}F\left(  \Phi\left(  x\right)
+\left(  0,t\right)  \right)  dt.
\]
Note that the bilinear version of this last equality is%
\[
\left\Vert \prod_{\mu=1}^{2}\widehat{F_{\mu}}\right\Vert _{L^{\frac{q}{2}%
}\left(  B\left(  0,R\right)  \right)  }^{\frac{q}{2}}=\left\Vert \prod
_{\mu=1}^{2}\mathcal{E}_{S}f_{\mu}\right\Vert _{L^{\frac{q}{2}}\left(
B\left(  0,R\right)  \right)  }^{\frac{q}{2}},\ \ \ \ \ \text{where }f_{\mu
}\left(  x\right)  \equiv\int_{-R^{-1}}^{R^{-1}}F_{\mu}\left(  \Phi\left(
x\right)  +\left(  0,t_{\mu}\right)  \right)  dt_{\mu}.
\]

Altogether then,%
\begin{align*}
&  \left\Vert \prod_{\mu=1}^{2}\mathcal{E}f_{\mu}\right\Vert _{L^{\frac{q}{2}%
}\left(  \alpha_{1}^{0}\right)  }^{\frac{q}{2}}=C_{\varepsilon}R^{\varepsilon
}R^{-1}\sum_{\alpha}\left\Vert \prod_{\mu=1}^{2}\mathcal{D}_{S}^{\Lambda_{\mu
}}a_{\mu}^{\alpha}\right\Vert _{L^{\frac{q}{2}}\left(  \alpha_{j_{1}}%
^{1}\right)  }^{\frac{q}{2}}=C_{\varepsilon}R^{\varepsilon}R^{-1}\sum_{\alpha
}\left\Vert \prod_{\mu=1}^{2}\mathcal{D}_{S}^{\Lambda_{\mu}}a_{\mu}^{\alpha
}\right\Vert _{L^{\frac{q}{2}}\left(  \alpha_{j_{1}}^{1}\right)  }^{\frac
{q}{2}}\\
&  =C_{\varepsilon}R^{\varepsilon}R^{-1}\sum_{\alpha}\left\Vert \prod_{\mu
=1}^{2}\widehat{F_{\mu}^{\alpha}}\right\Vert _{L^{\frac{q}{2}}\left(
\alpha_{j_{1}}^{1}\right)  }^{\frac{q}{2}}=C_{\varepsilon}R^{\varepsilon
}R^{-1}\sum_{\alpha}\left\Vert \prod_{\mu=1}^{2}\mathcal{E}_{S}f_{\mu}%
^{\alpha}\right\Vert _{L^{\frac{q}{2}}\left(  \alpha_{j_{1}}^{1}\right)
}^{\frac{q}{2}}=C_{\varepsilon}R^{\varepsilon}\sum_{\alpha}\left\Vert
\prod_{\mu=1}^{2}\mathcal{E}_{S}g_{\mu}^{\alpha}\right\Vert _{L^{\frac{q}{2}%
}\left(  \alpha_{j_{1}}^{1}\right)  }^{\frac{q}{2}},
\end{align*}
where $g_{\mu}^{\alpha}=R^{-\frac{1}{4}}f_{\mu}^{\alpha}$. Applying the
definition of the characteristic $\operatorname*{BiFour}\nolimits_{L^{q}%
\rightarrow L^{q}}^{S}\left(  R^{\frac{1}{2}}\right)  $, we thus obtain%
\[
\left\Vert \prod_{\mu=1}^{2}\mathcal{E}f_{\mu}\right\Vert _{L^{\frac{q}{2}%
}\left(  \alpha_{1}^{0}\right)  }^{\frac{q}{2}}\leq C_{\varepsilon
}R^{\varepsilon}\operatorname*{BiFour}\nolimits_{\otimes_{2}L^{q}\rightarrow
L^{\frac{q}{2}}}^{S}\left(  R^{\frac{1}{2}}\right)  \sum_{\alpha}\prod_{\mu
=1}^{2}\left\Vert g_{\mu}^{\alpha}\right\Vert _{L^{\frac{q}{2}}\left(
U\right)  }^{\frac{q}{2}}\ ,
\]
where%
\begin{align*}
g_{\mu}^{\alpha}\left(  x\right)   &  =R^{-\frac{1}{4}}\int_{-R^{-\frac{1}{2}%
}}^{R^{-\frac{1}{2}}}F_{\mu}^{\alpha}\left(  \Phi\left(  x\right)  +\left(
0,t_{\mu}\right)  \right)  dt_{\mu}=R^{-\frac{1}{4}}\int_{-R^{-\frac{1}{2}}%
}^{R^{-\frac{1}{2}}}\sum_{\lambda_{\mu}\in\Lambda_{\mu}\subset S}%
b_{\lambda_{\mu}}^{\alpha}\varphi_{R^{-\frac{1}{2}}}\left(  \Phi\left(
x\right)  +\left(  0,t_{\mu}\right)  +\Phi\left(  \lambda_{L_{\mu}}\right)
\right)  dt_{\mu}\\
&  =R^{-\frac{1}{4}}\int_{-R^{-\frac{1}{2}}}^{R^{-\frac{1}{2}}}\sum_{L_{\mu
}\in\mathcal{G}_{\frac{1}{2}s}\left[  U_{\mu}\right]  }e^{-ic_{\alpha}%
\cdot\Phi\left(  \lambda_{L_{\mu}}\right)  }\widehat{f_{L_{\mu}}}\left(
K_{L_{\mu},\alpha}\right)  \varphi_{R^{-\frac{1}{2}}}\left(  \Phi\left(
x\right)  +\left(  0,t_{\mu}\right)  +\Phi\left(  \lambda_{L_{\mu}}\right)
\right)  dt_{\mu}\\
&  =\sum_{L_{\mu}\in\mathcal{G}_{\frac{1}{2}s}\left[  U_{\mu}\right]
}e^{-ic_{\alpha}\cdot\Phi\left(  \lambda_{L_{\mu}}\right)  }\widehat
{f_{L_{\mu}}}\left(  K_{L_{\mu},\alpha}\right)  \left\{  R^{-\frac{1}{4}}%
\int_{-R^{-\frac{1}{2}}}^{R^{-\frac{1}{2}}}\varphi_{R^{-\frac{1}{2}}}\left(
\Phi\left(  x\right)  +\left(  0,t_{\mu}\right)  +\Phi\left(  \lambda_{L_{\mu
}}\right)  \right)  dt_{\mu}\right\}  \ .
\end{align*}

The approximate identity $\varphi_{R^{-\frac{1}{2}}}$\ is two-dimensional, so
that $\left\Vert \varphi_{R^{-\frac{1}{2}}}\right\Vert _{L^{\infty}}=R$, and
then the function%
\[
\psi_{L_{\mu}}\left(  x\right)  \equiv R^{-\frac{1}{4}}\int_{-R^{-\frac{1}{2}%
}}^{R^{-\frac{1}{2}}}\varphi_{R^{-\frac{1}{2}}}\left(  \Phi\left(  x\right)
+\left(  0,t_{\mu}\right)  +\Phi\left(  \lambda_{L_{\mu}}\right)  \right)
\]
`resembles' in the sense of Conclusion \ref{gather}, a smooth version of the
indicator $\mathbf{1}_{L_{\mu}}$ times $R^{-\frac{1}{4}}R^{-\frac{1}{2}%
}\left\Vert \varphi_{R^{-\frac{1}{2}}}\right\Vert _{L^{\infty}}=R^{\frac{1}%
{4}}$, and so is $L^{2}$-normalized, i.e. $\left\Vert \psi_{L_{\mu}%
}\right\Vert _{L^{2}}\approx1$.

We thus obtain in the case $q=4$ that%
\begin{align*}
&  \left\Vert g_{\mu}^{\alpha}\right\Vert _{L^{2}\left(  U\right)  }^{2}%
=\int_{U}\left\vert \sum_{L_{\mu}\in\mathcal{G}_{\frac{1}{2}s}\left[  U_{\mu
}\right]  }e^{-i\mathbf{c}_{\alpha}\cdot\Phi\left(  \lambda_{L_{\mu}}\right)
}\widehat{f_{L_{\mu}}}\left(  K_{L_{\mu},\alpha}\right)  \psi_{L_{\mu}}\left(
x\right)  \right\vert ^{2}dx\\
&  \approx\int_{U}\sum_{L_{\mu}\in\mathcal{G}_{\frac{1}{2}s}\left[  U_{\mu
}\right]  }\left\vert e^{-i\mathbf{c}_{\alpha}\cdot\Phi\left(  \lambda
_{L_{\mu}}\right)  }\widehat{f_{L_{\mu}}}\left(  K_{L_{\mu},\alpha}\right)
\psi_{L_{\mu}}\left(  x\right)  \right\vert ^{2}dx\\
&  =\sum_{L_{\mu}\in\mathcal{G}_{\frac{1}{2}s}\left[  U_{\mu}\right]
}\left\vert \widehat{f_{L_{\mu}}}\left(  K_{L_{\mu},\alpha}\right)
\right\vert ^{2}\int_{U}\left\vert \psi_{L_{\mu}}\left(  x\right)  \right\vert
^{2}dx\approx\sum_{L_{\mu}\in\mathcal{G}_{\frac{1}{2}s}\left[  U_{\mu}\right]
}\left\vert \widehat{f_{L_{\mu}}}\left(  K_{L_{\mu},\alpha}\right)
\right\vert ^{2},
\end{align*}
and hence%
\begin{align*}
\sum_{\alpha}\prod_{\mu=1}^{2}\left\Vert g_{\mu}^{\alpha}\right\Vert
_{L^{2}\left(  U\right)  }^{2}  &  \approx\sum_{\alpha}\prod_{\mu=1}^{2}%
\sum_{L_{\mu}\in\mathcal{G}_{\frac{1}{2}s}\left[  U_{\mu}\right]  }\left\vert
\widehat{f_{L_{\mu}}}\left(  K_{L_{\mu},\alpha}\right)  \right\vert ^{2}\\
&  \approx\sum_{\alpha}\sum_{\left(  K_{\mu}^{\pm},L_{\mu}^{\pm}\right)
\sim\alpha}\widehat{f_{L_{1}}}\left(  K_{L_{1}^{+},\alpha}\right)
\widehat{f_{L_{2}}}\left(  K_{L_{2}^{+},\alpha}\right)  \overline
{\widehat{f_{L_{1}^{-}}}\left(  K_{L_{1}^{-},\alpha}\right)  \widehat
{f_{L_{2}^{-}}}\left(  K_{L_{2}^{-},\alpha}\right)  },
\end{align*}
where we recall that $\left(  K,L\right)  \sim\alpha$ means $T_{K,L}\cap
\alpha\neq\emptyset$.

Now we exploit the fact that%
\begin{equation}
\text{for each fixed pair}\left(  L_{1},L_{2}\right)  \text{, the map }%
\alpha\rightarrow\left(  K_{L_{1},\alpha},K_{L_{2},\alpha}\right)  \text{ is
essentially one-to-one}.\label{1-1}%
\end{equation}
Indeed, if $\operatorname*{dist}\left(  \alpha,\alpha^{\prime}\right)  \gg
R^{\frac{1}{2}}$, then at least of the intervals $L_{\mu}$ must have
$\operatorname*{dist}\left(  K_{L_{\mu},\alpha},K_{L_{\mu},\alpha^{\prime}%
}\right)  \gg R^{\frac{1}{2}}$, as an inspection of the relevant diagram
reveals. This then allows us to apply the Cauchy-Schwarz inequality to obtain%
\begin{align}
& \label{to obtain}\\
&  \left\Vert \prod_{\mu=1}^{2}\mathcal{E}f_{\mu}\right\Vert _{L^{2}}%
^{2}\lesssim C_{\varepsilon}R^{\varepsilon}\operatorname*{BiFour}%
\nolimits_{L^{q}\rightarrow L^{q}}^{S}\left(  R^{\frac{1}{2}}\right)
\left\vert \sum_{\alpha}\sum_{\left(  K_{\mu}^{\pm},L_{\mu}^{\pm}\right)
\in\Omega_{\alpha}}\widehat{f_{L_{1}}}\left(  \lambda_{K_{1}^{+}}\right)
\widehat{f_{L_{2}}}\left(  \lambda_{K_{2}^{+}}\right)  \overline
{\widehat{f_{L_{1}^{-}}}\left(  \lambda_{K_{1}^{-}}\right)  \widehat
{f_{L_{2}^{-}}}\left(  \lambda_{K_{2}^{-}}\right)  }\right\vert \nonumber\\
&  \leq C_{\varepsilon}R^{\varepsilon}\operatorname*{BiFour}\nolimits_{L^{q}%
\rightarrow L^{q}}^{S}\left(  R^{\frac{1}{2}}\right)  \left(  \sum_{\alpha
}\sum_{\left(  K_{\mu}^{+},L_{\mu}^{+}\right)  \sim\alpha}\left\vert
\widehat{f_{L_{1}}}\left(  K_{1}^{+}\right)  \widehat{f_{L_{2}}}\left(
K_{2}^{+}\right)  \right\vert ^{2}\right)  ^{\frac{1}{2}}\left(  \sum_{\alpha
}\sum_{\left(  K_{\mu}^{-},L_{\mu}^{-}\right)  \sim\alpha}\left\vert
\widehat{f_{L_{1}^{\prime}}}\left(  K_{1}^{-}\right)  \widehat{f_{L_{2}%
^{\prime}}}\left(  K_{2}^{-}\right)  \right\vert ^{2}\right)  ^{\frac{1}{2}%
}\nonumber\\
&  =C_{\varepsilon}R^{\varepsilon}\operatorname*{BiFour}\nolimits_{L^{q}%
\rightarrow L^{q}}^{S}\left(  R^{\frac{1}{2}}\right)  \left(  \sum_{\left(
K_{1},L_{1}\right)  }\left\vert \widehat{f_{L_{1}}}\left(  K_{1}\right)
\right\vert ^{2}\right)  \left(  \sum_{\left(  K_{2},L_{2}\right)  }\left\vert
\widehat{f_{L_{2}}}\left(  K_{2}\right)  \right\vert ^{2}\right) \nonumber\\
&  =C_{\varepsilon}R^{\varepsilon}\operatorname*{BiFour}\nolimits_{L^{q}%
\rightarrow L^{q}}^{S}\left(  R^{\frac{1}{2}}\right)  \left\Vert
f_{1}\right\Vert _{L^{2}}^{2}\left\Vert f_{2}\right\Vert _{L^{2}}%
^{2}\ ,\nonumber
\end{align}
where in the last line we have used Plancherel's theorem on the torus. In
particular, we have proved that for every $\varepsilon>0$, there is a positive
constant $C_{\varepsilon}$ such that for $q\geq4$,
\begin{equation}
\operatorname*{BiFour}\nolimits_{L^{q}\rightarrow L^{\frac{q}{2}}}^{S}\left(
R\right)  \leq C_{\varepsilon}R^{\varepsilon}\operatorname*{BiFour}%
\nolimits_{L^{q}\rightarrow L^{\frac{q}{2}}}^{S}\left(  R^{\frac{1}{2}%
}\right)  ,\ \ \ \ \ \text{for all }R\gg1.\label{induction}%
\end{equation}

\subsection{Step 5: induction on scales}

Iterating the inequality (\ref{induction}) gives%
\begin{align*}
\operatorname*{BiFour}\nolimits_{\otimes_{2}L^{q}\rightarrow L^{\frac{q}{2}}%
}^{S}\left(  R\right)   &  \leq C_{\varepsilon}R^{\varepsilon}%
\operatorname*{BiFour}\nolimits_{\otimes_{2}L^{q}\rightarrow L^{\frac{q}{2}}%
}^{S}\left(  R^{\frac{1}{2}}\right) \\
&  \leq\left(  C_{\varepsilon}R^{\varepsilon}\right)  \left(  C_{\varepsilon
}R^{\frac{1}{2}\varepsilon}\right)  \operatorname*{BiFour}\nolimits_{\otimes
_{2}L^{q}\rightarrow L^{\frac{q}{2}}}^{S}\left(  R^{\frac{1}{4}}\right) \\
&  \vdots\\
&  \leq\left(  C_{\varepsilon}R^{\varepsilon}\right)  \left(  C_{\varepsilon
}R^{\frac{1}{2}\varepsilon}\right)  ...\left(  C_{\varepsilon}R^{\frac
{1}{2^{n-1}}\varepsilon}\right)  \operatorname*{BiFour}\nolimits_{\otimes
_{2}L^{q}\rightarrow L^{\frac{q}{2}}}^{S}\left(  R^{\frac{1}{2^{n}}}\right)  ,
\end{align*}
for all $n\in\mathbb{N}$ such that $R^{\frac{1}{2^{n}}}=2^{\frac{1}{2^{n}}%
s}\geq2$. If we choose $n$\ so that $\frac{1}{2^{n}}s\approx2$, then
$\operatorname*{BiFour}\nolimits_{\otimes_{2}L^{q}\rightarrow L^{\frac{q}{2}}%
}^{S}\left(  2^{\frac{1}{2^{n}}s}\right)  $ is dominated by a constant $C_{0}$
independent of $s$, and $n\approx\log_{2}s$, which implies%
\begin{align*}
&  \left(  C_{\varepsilon}R^{\varepsilon}\right)  \left(  C_{\varepsilon
}R^{\frac{1}{2}\varepsilon}\right)  ...\left(  C_{\varepsilon}R^{\frac
{1}{2^{n-1}}\varepsilon}\right)  =\left(  C_{\varepsilon}2^{\varepsilon
s}\right)  \left(  C_{\varepsilon}2^{\frac{1}{2}\varepsilon s}\right)
...\left(  C_{\varepsilon}2^{\frac{1}{2^{n-1}}\varepsilon s}\right) \\
&  \leq C_{\varepsilon}^{\log_{2}s}2^{2\varepsilon s}=2^{\log_{2}%
C_{\varepsilon}\log_{2}s}2^{2\varepsilon s}=s^{\log_{2}C_{\varepsilon}%
}2^{2\varepsilon s}\ll C_{\varepsilon}2^{3\varepsilon s}.
\end{align*}

Thus we conclude that for every $\varepsilon^{\prime}>0$, there is a positive
constant $C_{\varepsilon^{\prime}}$ such that%
\[
\operatorname*{BiFour}\nolimits_{\otimes_{2}L^{q}\rightarrow L^{\frac{q}{2}}%
}^{S}\left(  R\right)  \leq C_{\varepsilon^{\prime}}R^{\varepsilon^{\prime}%
},\ \ \ \ \ \text{for all }s>1,
\]
which is a local consequence of the bilinear inequality of Bennett, Carbery
and Tao \cite{BeCaTa},
\[
\left\Vert \prod_{\mu=1}^{2}\mathcal{E}f_{\mu}\right\Vert _{L^{2}}%
\lesssim\prod_{\mu=1}^{2}\left\Vert f_{\mu}\right\Vert _{L^{2}}\ ,
\]
and using \textbf{Step 1}, this completes the second argument that the Fourier
extension theorem for the parabola in the plane holds.


\begin{thebibliography}{999999}                                                                                           %
\bibitem[Alp]{Alp}\textsc{B. K. Alpert, }\textit{A class of bases in} $L^{2} $
\textit{\ for the sparse representation of integral operators}, SIAM J. Math.
Anal \textbf{1} (1993), p. 246-262.

\bibitem[BeCaTa]{BeCaTa}\textsc{J. Bennett, A. Carbery and T. Tao}, \textit{On
the multilinear restriction and Kakeya conjectures}, Acta. Math. \textbf{196}
(2006), 261-302.

\bibitem[Bou]{Bou}\textsc{J. Bourgain,} \textit{Some new estimates for
oscillatory integrals}, in "Essays on Fourier analysis in honor of Elias M.
Stein", ed. C.\ Fefferman, R. Fefferman and S. Wainger, Princeton University
Press, 1994.

\bibitem[BoGu]{BoGu}\textsc{J. Bourgain and L. Guth},\textit{\ Bounds on
oscillatory integral operators based on multilinear estimates},
\texttt{arXiv:1012.3760v3}.

\bibitem[Bus]{Bus}\textsc{S. Buschenhenke}, \textit{Factorization in Fourier
restriction theory and near extremizers}, Math. Nachr. \textbf{297 }(2024), 195--208.

\bibitem[CaSj]{CaSj}\textsc{L. Carleson and P. Sj\"{o}lin,}
\textit{Oscillatory integrals and a multiplier problem for the disc}, Studia
Math. \textbf{44} (1972), 287--299.

\bibitem[Dem]{Dem}\textsc{C. Demeter,} \textit{Fourier restriction, decoupling
and applications,} Cambridge Studies in Advanced Mathematics \textbf{184}, 2020.

\bibitem[Fef]{Fef}\textsc{C. Fefferman, }\textit{Inequalities for strongly
singular convolution operators}, Acta Math. \textbf{124} (1970), 9--36.

\bibitem[Ste]{Ste}\textsc{E. M. Stein,} \textit{Some problems in harmonic
analysis}, Harmonic analysis in Euclidean spaces (Proc. Sympos. Pure Math.,
Williams Coll., Williamstown, Mass., 1978), Part 1, pp. 3-20, Proc. Sympos.
Pure Math., \textbf{XXXV}, Part, Amer. Math. Soc., Providence, R.I., 1979.

\bibitem[Ste2]{Ste2}\textsc{E. M. Stein,} \textit{Harmonic Analysis:
real-variable methods, orthogonality, and oscillatory integrals}%
,\textit{\ }Princeton University Press, Princeton, N. J., 1993.

\bibitem[Tao1]{Tao1}\textsc{T. Tao}, \textit{The B\^{o}chner-Riesz conjecture
implies the restriction conjecture}, Duke Math. J. \textbf{96} (1999), no. 2, 363-375.

\bibitem[TaVaVe]{TaVaVe}\textsc{T. Tao, A. Vargas, and L. Vega}, \textit{A
bilinear approach to the restriction and Kakeya conjectures}, JAMS \textbf{11}
(1998), no. 4, 967--1000.

\bibitem[Zyg]{Zyg}\textsc{A. Zygmund,} \textit{On Fourier coefficients and
transforms of functions of two variables}, Studia Mathematica \textbf{50}
(1974), no. 2, 189-201.
\end{thebibliography}
\end{document}